\documentclass[11pt]{amsart}
\usepackage[margin=1.49in]{geometry}

\usepackage{amssymb, amsmath, color}
\usepackage{amsthm}
\usepackage{mathrsfs}
\usepackage{cite}
\usepackage{hyperref}
\usepackage{enumitem}
\usepackage{cases}
\usepackage[font={it}, margin=1cm]{caption}
\usepackage{verbatim}
\usepackage{algorithmicx}
\usepackage{algorithm} 
\usepackage[noend]{algpseudocode}

\newtheorem{theorem}{Theorem}[section]
\newtheorem{question}[theorem]{Question}
\newtheorem{lemma}[theorem]{Lemma} 
\newtheorem{proposition}[theorem]{Proposition} 
\newtheorem{thmletter}{Theorem}

\newtheorem{corolletter}[thmletter]{Corollary}
\newtheorem{corollary}[theorem]{Corollary} 
\newtheorem{definition}[theorem]{Definition}

\newtheorem{example}[theorem]{Example}

\newtheoremstyle{remark}{3pt}{3pt}{}{10pt}{\itshape}{.}{.5em}{}
\theoremstyle{remark}
\newtheorem{remark}[theorem]{Remark}

\newcommand{\p}[1]{\noindent {\newline\bfseries #1.}}
\newcommand{\emo}{\operatorname{End}}
\newcommand{\out}{\operatorname{Out}}
\newcommand{\aut}{\operatorname{Aut}}
\newcommand{\inn}{\operatorname{Inn}}

\newcommand{\fix}{\operatorname{Fix}}
\newcommand{\im}{\operatorname{Im}}
\newcommand{\rk}{\operatorname{rank}}

\newcommand{\MOFix}{\operatorname{MOFix}}

\title[Fixed points of endomorphisms for $F_2$]{Fixed points and stable images of endomorphisms for the free group of rank two}
\author{Laura Ciobanu} 
\address{School of Mathematical and Computer Sciences,
 Heriot--Watt University, 
 Edinburgh EH14 4AS,
 Scotland}
\email{L.Ciobanu@hw.ac.uk}

\author{Alan D. Logan}
\address{School of Mathematical and Computer Sciences,
 Heriot--Watt University, 
 Edinburgh EH14 4AS,
 Scotland}
\email{A.Logan@hw.ac.uk}



\begin{document}
\maketitle

\begin{abstract}
We give an algorithm which computes the fixed subgroup and the stable image for any endomorphism of the free group of rank two $F_2$, answering for $F_2$ a question posed by Stallings in 1984 and a question of Ventura.
\end{abstract}

\section{Introduction}
\label{introduction}
Let $F$ be a finitely generated free group, and let $\emo(F)$ and $\aut(F)$ be the set of endomorphisms and automorphisms of $F$, respectively. For an endomorphism $\psi \in \emo(F)$, the fixed subgroup $\fix(\psi)$ of $\psi$ is defined as
\[
\fix(\psi):=\{x\in F\mid\psi(x)=x\}.
\]
Fixed subgroups generated a wide body of work in the 1970s--1990s in relation to the ``Scott conjecture'' (see \cite{Ventura2002Fixed}), which stated that such subgroups are finitely generated, so $\rk(\fix(\psi))<\infty$. For $\psi \in \aut(F)$, Gersten showed that $\fix(\psi)$ is finitely generated \cite{Gersten1987Fixed}, and then in their seminal paper investigating Thurston's train track maps \cite{Bestvina1992Traintracks}, Bestvina and Handel strengthened Gersten's result by proving that $\rk(\fix(\psi)) \leq \rk(F)$. Imrich and Turner extended this to show that $\rk(\fix(\psi))\leq\rk(F)$ for all $\psi \in \emo(F)$ \cite{Imrich1989Endomorphisms}.

While the above results elucidated the rank bounds for a fixed subgroup, finding the precise rank or the generators in the case of an arbitrary endomorphism remains so far intractable and the question posed by Stallings in 1984 is still open \cite[Problems P3 \& 5]{Stallings1987Graphical}.
\begin{question}[Stallings, 1984]
\label{Qn:Stallings}
Does there exist an algorithm with input an endomorphism $\psi\in\emo(F)$ and output $\rk(\fix(\psi))$?
\end{question}
For $\psi\in\aut(F)$, this question was resolved positively, and moreover an algorithm with output a \emph{basis} for $\fix(\psi)$ was shown to exist \cite{Bogopolski2016algorithm}
(see also
\cite{Feighn2018algorithmic}).

The main result of this paper provides a positive answer to Question \ref{Qn:Stallings} for any endomorphism $\psi\in\emo(F_2)$, so when $F=F_2$ is the free group of rank two, again by giving a basis for $\fix(\psi)$.

\begin{thmletter}
\label{thm:basisexists}
There exists an algorithm with input an endomorphism $\psi\in\emo(F_2)$ and output a basis for $\fix(\psi)$.
\end{thmletter}

\p{Stable images} A consequence of Theorem \ref{thm:basisexists} is that we can to compute the \emph{stable image} of $\psi\in\operatorname{End}(F)$, which is the subgroup $\psi^{\infty}(F):=\cap_{i=1}^{\infty}\psi^i(F)$ of $F$. This is the key object
used in Imrich and Turner's paper mentioned above,
and the computation of its basis is an open question of Ventura \cite[Abstract 3.24, Problem 4.6]{Dagstuhl2019}. 

\begin{corolletter}
\label{corol:stableimage}
There exists an algorithm with input an endomorphism $\psi\in\emo(F_2)$ and output a basis for $\psi^{\infty}(F_2)$.
\end{corolletter}

\p{Endomorphisms versus automorphisms}
There is a rich theory of automorphisms of free groups, but the theory of endomorphisms of free groups is far less developed.
This ``endomorphism vs automorphism'' disparity is one of the reasons our paper is longer and more involved than one might expect for $F_2$.

For example, if the ``endomorphism-twisted conjugacy problem'' for $F_2$ is shown to be decidable in the future, then Section \ref{sec:solvingInstances} can be omitted. Meanwhile, although the automorphism-twisted conjugacy problem is decidable \cite{Bogopolski2006Conjugacy}, has been investigated extensively, and generalised to other groups \cite{Ladra2011generalized} \cite{sankaran2020twisted} \cite{gonccalves2020twisted}, there is only sparse literature on the endomorphism-twisted conjugacy problem \cite{Kim2016Twisted}.

Another example of this disparity concerns train track maps \cite{Bestvina1992Traintracks}: these have been used for around 30 years as the topological setting for free group automorphisms, and were essential in the answer of Question \ref{Qn:Stallings} for automorphisms. On the other hand, train track maps for endomorphisms have only recently been considered, originally in an unpublished preprint of Reynolds \cite{reynolds2010dynamics}, then by Mutanguha \cite{Mutanguha2018Hyperbolic} \cite{mutanguha2019irreducible} \cite{mutanguha2020dynamics}, and this work of Mutanguha gives us the central algorithm of our paper.

\noindent{\newline\bfseries The free group of rank two?}
Throughout this paper we use properties unique to the free group of rank two. For example, the description of primitive elements of $F_2$ (elements that belong to a basis of $F_2$) due to Cohen, Metzler and Zimmerman \cite{Cohen1981WhatDoes} is fundamental to Section \ref{sec:classify}, while Lemma \ref{lem:algo_case1} applies both this description and the classical result of Nielsen that $\out(F_2)\cong \operatorname{GL}_2(\mathbb{Z})$ under the natural map. More fundamentally, this paper is principally concerned with non-surjective endomorphisms $\psi\in\emo(F_2)$, and here $\fix(\psi)$ is infinite cyclic or trivial, so the problem of finding a basis for $\fix(\psi)$ is equivalent to determining if $\fix(\psi)$ is trivial or not. This simplification fails in free groups of higher rank.

\p{Outline of the paper}
In Section \ref{sec:applications} we prove Corollary \ref{corol:stableimage}, which follows from Theorem \ref{thm:basisexists} together with results of Imrich and Turner.

The proof of Theorem \ref{thm:basisexists} relies on several observations and results. It turns out that the main case to consider is when the endomorphism $\psi$ is injective but not surjective, as we explain in Section \ref{sec:prelim}. In this case we have the simplification which drives the rest of the paper: the fixed subgroup, if non-trivial, is generated by a single primitive element (this is Lemma \ref{lem:fixedprimitives}).

In order to find this element we take a long detour and first find fixed points up to conjugacy, that is, the \emph{(maximal) outer fixed points}.
An outer fixed point of $\phi$ is a conjugacy class consisting of elements which are mapped to a conjugate by $\phi$; these conjugacy classes behave similarly to fixed points, but there are important differences. In particular, maximal (maximality means not being a proper power) fixed points of non-surjective monomorphisms are unique, but maximal outer fixed points may not be (see Example \ref{ex:4MOFP}).
The bulk of this paper (Sections \ref{sec:ofp1} -- \ref{sec:classify}) is devoted to proving the following.

\begin{thmletter} \label{thm:maxoutalgo}
There exists an algorithm with input a non-surjective monomorphism $\psi \in \emo(F(a,b))$ and output the maximal outer fixed points of $\psi$.
\end{thmletter}

Theorem \ref{thm:maxoutalgo} is used as follows: if $\fix(\psi) \neq \{1\}$ then we extract the generator(s) of $\fix(\psi)$ from one of the finitely many conjugacy classes output by the algorithm in Theorem \ref{thm:maxoutalgo}. This algorithm relies on testing which one of two conditions the mapping torus of $\psi$ satisfies (see the proof of Proposition \ref{prop:KapMut}), and while this process will terminate, it will not give an efficient algorithm.

In Proposition \ref{prop:maxoutalgo} we split Theorem \ref{thm:maxoutalgo} into two cases, depending on the ``exponent-sum matrix'' $\Psi_{(a,b)}$ of $\psi\in\emo(F(a, b))$. When $\det(\Psi_{(a,b)})\neq \pm 1$ we employ elementary linear algebra and the classical Nielsen theory of automorphisms of $F_2$ to find the outer fixed points (Section \ref{sec:ofp1}), and when $\det(\Psi_{(a,b)})=\pm1$ we have a much more involved argument (Sections \ref{sec:ofp2} and \ref{sec:classify}).

In Section \ref{sec:ofp2} we assume $\det(\Psi_{(a,b)})=\pm1$ and use an algorithm, based on ideas and results of Mutanguha and Kapovich, which determines the hyperbolicity of the mapping torus of $\psi$. This algorithm allows us to determine whether or not some power $\psi^k$ of $\psi$ has non-trivial outer fixed points.
In Section \ref{sec:classify} we classify the outer fixed points and show, in Theorem \ref{thm:OuterFPClassification}, that there are at most two maximal outer fixed points (up to inversion) for a non-surjective monomorphism $\psi$ that satisfies $\det(\Psi_{(a,b)})=\pm1$. We end the section by combining the above algorithms and classification results to prove Theorem \ref{thm:maxoutalgo}, which says that these points are computable.

Finally, in Section \ref{sec:FixedPoints} we use the knowledge of the outer fixed points of $\psi$ to compute a basis for $\fix(\psi)$, and in particular we prove Theorem \ref{thm:basisexists}. This involves resolving a special case of the endomorphism-twisted conjugacy problem for $F_2$.

\section*{Acknowledgements}
The authors were supported by EPSRC Standard Grant EP/R035814/1. The first-named author would like to thank the organisers of the Dagstuhl seminar 19131 \emph{Algorithmic Problems in Group Theory}, where the topics addressed in this paper were discussed and listed as important open questions in the theory of free groups \cite[Abstract 3.24, Problem 4.6]{Dagstuhl2019}.

\section{Computing stable images in $F_2$}
\label{sec:applications}
We start the paper by proving Corollary \ref{corol:stableimage}.
Recall from the introduction that the {stable image} of $\psi\in\operatorname{End}(F)$ is the subgroup $\psi^{\infty}(F):=\cap_{i=1}^{\infty}\psi^i(F)$ of $F$. This subgroup was the key object studied in \cite{Imrich1989Endomorphisms}, where they observed that a monomorphism $\psi$ acts as an automorphism on $\psi^{\infty}(F)$.

In general, finding a basis for $\psi^{\infty}(F)$ is harder than for $\fix(\psi)$, in the sense that given a basis for $\psi^{\infty}(F)$ one can compute a basis for $\fix(\psi)$:
If $\psi$ is injective then $\psi$ acts as an automorphism on $\psi^{\infty}(F)$, so $\psi|_{\psi^{\infty}(F)}\in\aut(\psi^{\infty}(F))$, and moreover $\fix(\psi)=\fix(\psi|_{\psi^{\infty}(F)})$, so apply the algorithm for automorphisms \cite{Bogopolski2016algorithm} to $\psi|_{\psi^{\infty}(F)}$. If $\psi$ is not injective, then this follows from the injective case \cite[Theorem 2]{Imrich1989Endomorphisms}.
For $F_2$, Theorem \ref{thm:basisexists} allows us go the other way, giving Corollary \ref{corol:stableimage}.

\begin{proof}[Proof of Corollary \ref{corol:stableimage}]
Note that we can algorithmically determine if $\psi$ is surjective or not (for example, by using Stallings' foldings). If $\psi$ is surjective then $\psi^{\infty}(F_2)=F_2$ and the basis for $\psi^{\infty}(F_2)$ is the basis of $F_2$.

So suppose $\psi$ is non-surjective. Then $\psi^{\infty}(F_2)$ is a proper retract of $F_2$ \cite[Theorem 1]{Turner1996TestWords}, and hence $\psi^{\infty}(F_2) \cong \{1\}$ or $\psi^{\infty}(F_2) \cong \mathbb{Z}$ (see page 103 in \cite{mks} for more on retracts of $F_2$). Now, $\psi$ acts as an automorphism on $\psi^{\infty}(F_2)$ \cite{Imrich1989Endomorphisms}, and since all automorphisms of $\mathbb{Z}$ or $\{1\}$ are involutions, $\psi^2$ acts trivially on $\psi^{\infty}(F_2)$. Hence,
$\psi^{\infty}(F_2)\leq\fix(\psi^2)$. Note also that $\fix(\psi^2)\leq(\psi^2)^{\infty}(F_2)$ and that $\psi^{\infty}(F_2)=(\psi^2)^{\infty}(F_2)$. These three facts combine to give that $\fix(\psi^2)=\psi^{\infty}(F_2)$. By Theorem \ref{thm:basisexists}, we can compute a basis for $\fix(\psi^2)$, and hence also for $\psi^{\infty}(F_2)$.
\end{proof}

\section{Preliminaries on fixed points}
\label{sec:prelim}

In this section we give some preliminary definitions, remarks and lemmas which underlie this paper. We also introduce the notion of an ``outer fixed point''; the bulk of this paper is devoted to the study of outer fixed points.

\p{Input to the algorithm}
In this paper we give an algorithm with input an endomorphism $\psi\in\emo(F_2)$ and output a basis for $\fix(\psi)=\{x\in F_2\mid \psi(x)=x\}$, where the input satisfies the following points.
 Firstly, we usually work with a fixed basis $a, b$ for $F_2$, and write $F(a, b)$ for the free group with this basis. This basis is implicitly part of the input. Secondly, if we are given an endomorphism $\psi: F(a, b)\rightarrow F(a, b)$ then we are either given both $\psi(a)$ and $\psi(b)$ explicitly in the input, or these words can be pre-computed.

\p{The three cases of Theorem \ref{thm:basisexists}}
Since free groups are Hopfian, every surjective endomorphism is an automorphism. The algorithm of Theorem \ref{thm:basisexists}, which computes the basis for $\fix(\psi)$, therefore splits into three cases:

\begin{enumerate}[label=(\roman*)]
\item\label{list:Aut} $\psi$ is an automorphism. Theorem \ref{thm:basisexists} is known to hold for this case \cite{Bogopolski2016algorithm} (see also \cite{Feighn2018algorithmic}).
\item\label{list:injNotOnto} $\psi$ is injective but not surjective. Here, $\fix(\psi)$ is either trivial or infinite cyclic \cite{Turner1996TestWords}. This case is the core of the present paper.
\item\label{list:NotInjNotOnto} $\psi$ is neither injective nor surjective. Here $\fix(\psi)$ is again either trivial or infinite cyclic. We prove this case in Lemma \ref{lem:nonInjCase}, below.
\end{enumerate} 

We first resolve Case \ref{list:NotInjNotOnto} of this list.

\begin{lemma}
\label{lem:nonInjCase}
There exists an algorithm with input a non-injective endomorphism $\psi: F(a, b)\rightarrow F(a, b)$ and output a basis for $\fix(\psi)$.
\end{lemma}

\begin{proof}
As $\psi$ is non-injective we have that $\im(\psi)$ is cyclic (possibly trivial). A generating element $w$ of $\im(\psi)$ can therefore be found by standard algorithms (for example, Stallings' folding algorithm). Suppose that $\fix(\psi)$ is non-trivial. Then there exists some $k\in\mathbb{Z}\setminus\{0\}$ such that $\psi(w^k)=w^k$. Then $\psi(w)^k=w^k$, and as roots are unique in free groups we have that $\psi(w)=w$. Therefore, $\fix(\psi)$ is non-trivial if and only if $\im(\psi)=\langle w\rangle$ is non-trivial and $\psi(w)=w$, if and only if $\im(\psi)$ is non-trivial and $\fix(\psi)=\im(\psi)$.

As the generator $w$ was computed algorithmically, and as we can algorithmically decide if $\psi(w)$ is equal to $w$ or not, the algorithm to compute a basis for $\fix(\psi)$ is as follows: First obtain a generator $w$ for $\im(\psi)$ using standard algorithms. Then compute $\psi(w)$. If $\psi(w)=w$ and $w\neq1$ then $\fix(\psi)=\langle w\rangle$. Else, $\fix(\psi)=\{1\}.$
\end{proof}

\p{Non-surjective monomorphisms}
As Cases \ref{list:Aut} and \ref{list:NotInjNotOnto} are resolved, we focus on Case \ref{list:injNotOnto}, so on non-surjective monomorphisms. We begin with two preliminary results on fixed points of such maps. By a \emph{primitive word (or element)} of a free group $F$ we mean a word $x\in F$ such that $x$ is an element of a free basis of $F$.

\begin{lemma}
\label{lem:fixedprimitives}
Let $\psi: F(a, b)\rightarrow F(a, b)$ be a non-surjective monomorphism. If $\psi(w)=w$ then there exists a primitive element $x$ such that $w=x^i$ and $\psi(x)=x$.
\end{lemma}

\begin{proof}
As $\psi$ is a non-surjective monomorphism, the stable image $\psi^{\infty}(F(a, b))$ is a free factor of $F(a, b)$ \cite[Theorem 1]{Turner1996TestWords}, and hence is cyclic and generated by a primitive element $x$ of $F(a, b)$. As $\fix(\psi)\leq\psi^{\infty}(F(a, b))$, the fixed element $w$ is a power of $x$, so $w=x^i$. As $\psi(x^i)=x^i$, and as roots are unique in free groups, we have that $\psi(x)=x$.
\end{proof}

A word ${w}\in F(a, b)$ is a \emph{proper power} if there exists some ${w}_0\in F(a, b)$ and some integer $p>1$ such that ${w}={w}_0^p$.
Define a \emph{maximal fixed point} of an endomorphism $\psi: F\rightarrow F$ to be a fixed point $x\in\fix(\psi)$ such that $x$ is not a proper power.
Note that for the word ``maximal'' to make sense every fixed point should be a power of a maximal fixed point.
This is indeed the case, as if $w$ is a fixed point of $\psi$ then it is contained in a maximal cyclic subgroup $\langle x\rangle$ of $F$, and as in the proof of Lemma \ref{lem:fixedprimitives} we see that $x$ is a fixed point also.
We then have the following:

\begin{lemma}
\label{lem:fpUnique}
For $\psi: F(a, b)\rightarrow F(a, b)$ a non-surjective monomorphism we have $\fix(\psi)=\{1\}$ or $\fix(\psi)=\langle x \rangle$, where $x$ is a primitive element. Moreover, $x$ and $x^{-1}$ are the only maximal fixed points of $\psi$.
\end{lemma}

\begin{proof}
If $\fix(\psi)$ contains a non-trivial fixed point then it contains a maximal one. Let $x$ and $y$ be two non-equal maximal fixed points. Suppose $x$ and $y$ do not commute. Then $\psi([x, y])=[x, y]\neq1$. Hence, by Lemma \ref{lem:fixedprimitives}, $[x, y]$ is the power of a primitive element, a contradiction (to see this contradiction note that, for example, no primitive element is killed by the abelianisation map). Therefore, $\langle x, y\rangle$ is cyclic. By maximality and as $x\neq y$, we have $x=y^{-1}$. Hence, $x$ and $x^{-1}$ are the only maximal fixed points of $\psi$. Therefore, every fixed point has the form $x^i$ for some $i\in\mathbb{Z}$, and the result follows.
\end{proof}

\p{Outer fixed points}
 Our strategy to compute $\fix(\psi)$ for $\psi$ a non-surjective monomorphism is by computing the \emph{outer fixed points} of $\psi$, defined below.
\begin{definition}
Let $[{w}]$ denote the conjugacy class of ${w}\in F(a,b)$ (to avoid confusion with citations, we often use $[1]_F$ to denote the conjugacy class of the identity element $1\in F(a, b)$). We write ${g}\sim_{G}{h}$ to mean that ${g}$ and ${h}$ are conjugate in the group $G$ and $g\nsim_Gh$ to mean that they are not conjugate, and we write ${g}\sim{h}$ and ${g}\nsim{h}$ when the group $G$ is understood.
\begin{enumerate}[label=(\roman*)]
\item An \emph{outer fixed element} of an endomorphism $\psi\in\emo(F(a, b))$ is an element ${w}\in F(a, b)$ such that $\psi({w})\sim{w}$. A \emph{maximal outer fixed element} is an outer fixed element which is not a proper power.
\item An \emph{outer fixed point} of an endomorphism $\psi\in\emo(F(a, b))$ is the conjugacy class $[{w}]$ of an outer fixed element ${w}$ of $\psi$. Note that $[{w}]$ satisfies $\psi([{w}]) \subseteq [{w}]$. A \emph{maximal outer fixed point} is the conjugacy class $[{w}]$ of a maximal outer fixed element ${w}$.
\item The \emph{trivial outer fixed point}, denoted $[1]_F$, is the conjugacy class of the identity element $1\in F(a, b)$. Note that $[1]_F$ is always an outer fixed point, but never a maximal one.
\end{enumerate}
We shall use $\MOFix(\psi)$ to to denote the set of maximal outer fixed points of the endomorphism $\psi$. If $\alpha$ is a conjugacy class then we write $\alpha^{-1}$ to mean the conjugacy class $\{x^{-1}\mid x\in\alpha\}$.
\end{definition}

In order to describe all outer fixed points of a non-surjective monomorphism it is sufficient to find all the maximal outer fixed points.

If $\alpha$ is an outer fixed point of $\psi$ then for any ${u}\in\alpha$ there exists some $\gamma\in\inn(F)$ such that ${u}\in\fix(\psi\gamma)$, and if in addition $\alpha$ is a maximal outer fixed point of $\psi$ then ${u}$ is a maximal fixed point of $\psi\gamma$.
Hence, by Lemma \ref{lem:fixedprimitives}, we get:
\begin{lemma}\label{lem:maxouter}
Maximal outer fixed elements of non-surjective monomorphisms of $F(a,b)$ are primitive elements.
Maximal outer fixed points of non-surjective monomorphisms of $F(a,b)$ are conjugacy classes of primitive elements.
\end{lemma}

\p{The associated matrix of an endomorphism}
For a word ${w}\in F(a, b)$ let $\sigma_a({w})$ be the exponent-sum of the $a$-terms in ${w}$, and define $\sigma_b({w})$ analogously.
The \emph{associated matrix $\Psi_{(a,b)}\in M_{2\times2}(\mathbb{Z})$} of $\psi: F(a, b)\rightarrow F(a, b)$, $a\mapsto A, b\mapsto B$ is then

\begin{equation}\label{matrix}
\Psi_{(a,b)}:=\left(\begin{array}{cc} \sigma_a(A)&\sigma_b(A)\\\sigma_a(B)&\sigma_b(B)\end{array}\right),
\end{equation}
and this matrix defines the action of $\psi$ on the abelianisation of $F(a, b)$.

Our approach to proving Theorem \ref{thm:maxoutalgo}, that is, to finding outer fixed points of a monomorphism $\psi$, is based on the properties of its associated matrix. This approach splits into two cases, which we state in the following proposition, used essentially as a referencing tool.
\begin{proposition} \label{prop:maxoutalgo}
The algorithm in Theorem \ref{thm:maxoutalgo} to compute the maximal outer fixed points of a non-surjective monomorphism $\psi:F(a, b)\rightarrow F(a, b)$ splits into two cases:
\begin{enumerate}[label=(\Roman*)]
\item\label{maxoutalgo:I} $\det(\Psi_{(a,b)})\neq\pm1$, and
\item\label{maxoutalgo:II} $\det(\Psi_{(a,b)})=\pm1$. 
\end{enumerate} 
\end{proposition}
The proof of Theorem \ref{thm:maxoutalgo} takes all of Sections \ref{sec:ofp1} -- \ref{sec:classify}. Case \ref{maxoutalgo:II} follows from Theorem \ref{thm:OuterFPClassification} and is significantly more difficult than Case \ref{maxoutalgo:I}, which follows from Lemma \ref{lem:algo_case1}. We combine the cases and prove Theorem \ref{thm:maxoutalgo} at the end of Section \ref{sec:classify}.

\section{Proof for Case \ref{maxoutalgo:I} of Proposition \ref{prop:maxoutalgo}}
\label{sec:ofp1}
Here we consider Case \ref{maxoutalgo:I} of Proposition \ref{prop:maxoutalgo}, so non-surjective monomorphisms $\psi: F(a, b)\rightarrow F(a, b)$ with $\det(\Psi_{(a,b)})\neq\pm1$.
In this case, we prove that maximal outer fixed points are unique up to inversion.

We start with a preliminary lemma on matrices.

\begin{lemma}
\label{lem:Matrices}
Let $\Psi, M \in M_{2\times2}(\mathbb{Z})$ be such that $\det(\Psi)\neq\pm1$ and $M\Psi=M$.

If $M=\left(\begin{array}{cc}p_1 & q_1\\ p_2&q_2 \end{array}\right)$ satisfies $\gcd(p_1, q_1)=1=\gcd(p_2, q_2)$, then either $(p_1, q_1)=(p_2, q_2)$ or $(p_1, q_1)=-(p_2, q_2)$.
\end{lemma}
\begin{proof}
If $\det(M)\neq0$ then $\Psi$ is the identity matrix, contradicting $\det(\Psi)\neq\pm1$. Hence, $\det(M)= 0,$ so $q_1p_2=p_1q_2$.

Suppose that all of $p_1, q_1, p_2, q_2$ are non-zero. As $\gcd(p_1, q_1)=1=\gcd(p_2, q_2)$ and $q_1p_2=p_1q_2$, either $(p_1, q_1)=(p_2, q_2)$ or $(p_1, q_1)=-(p_2, q_2)$, as required.

Suppose that one of $p_1, q_1, p_2, q_2$ is $0$, and without loss of generality we may assume that either $p_1=0$ or $q_1=0$.
If $p_1=0$ then as $\gcd(p_1, q_1)=1$ we have $q_1=\epsilon_1$ for $\epsilon_1=\pm1$. As $q_1p_2=p_1q_2$ we further have that $\epsilon_1p_2=0$, and so $p_2=0$. This in turn implies that $q_2=\epsilon_2$ for some $\epsilon_2=\pm1$. Hence, $p_1=0=p_2$ and $|q_1|=1=|q_2|$, and so either $(p_1, q_1)=(p_2, q_2)$ or $(p_1, q_1)=-(p_2, q_2)$ as required. Using identical reasoning, if $q_1=0$ then $|p_1|=1=|p_2|$ and $q_1=0=q_2$, and the result follows.
\end{proof}

We now apply Lemma \ref{lem:Matrices} to outer fixed points. Note that as $\det(\Psi_{(a,b)})\neq\pm1$, the endomorphisms $\psi$ covered by Lemma \ref{lem:useAb} are non-surjective.

\begin{lemma}
\label{lem:useAb}
Let $\psi:F(a, b)\rightarrow F(a, b)$ be a monomorphism with $\det(\Psi_{(a,b)})\neq\pm1$.

If $\alpha$ and $\beta$ are maximal outer fixed points of $\psi$ then either $\alpha=\beta$ or $\alpha=\beta^{-1}$.
\end{lemma}

\begin{proof}
Let $\pi:F(a,b) \mapsto \mathbb{Z}^2$ be the abelianisation map.
Let $x\in\alpha$ and $y\in\beta$, where $\alpha, \beta$ are as in the statement of the lemma. Write $(p_1, q_1):=\pi(x)$ and $(p_2, q_2):=\pi(y)$.
Since $x$ and $y$ are contained in maximal outer fixed points they satisfy $(p_1, q_1)\Psi_{(a,b)}=(p_1, q_1)$ and $(p_2, q_2)\Psi_{(a,b)}=(p_2, q_2)$; moreover, by Lemma \ref{lem:maxouter} they are primitive elements and so $\gcd(p_1, q_1)=1=\gcd(p_2, q_2)$.
Lemma \ref{lem:Matrices} is therefore applicable, and so either $\pi(x)=\pi(y)$ or $\pi(x)=-\pi(y)$.
The result follows as conjugacy classes of primitive elements of $F(a, b)$ are uniquely defined by their images in the abelianisation \cite[Corollary N4]{mks}.
\end{proof}

Lemmas \ref{lem:maxouter} and \ref{lem:useAb} give the following description of (maximal) outer fixed points of the endomorphisms from Case \ref{maxoutalgo:I} of Proposition \ref{prop:maxoutalgo}.
\begin{corollary} \label{corol:outercase1}
If $\psi: F(a, b)\rightarrow F(a, b)$ is a monomorphism with $\det(\Psi_{(a,b)})\neq \pm 1$ then either $\psi$ has $[1]_F$ as its unique outer fixed point, or $\psi$ has two maximal outer fixed points $[x]$ and $[x^{-1}]$, where $x$ is a primitive element. In the latter case every outer fixed point of $\psi$ has the form $[x^i]$ for some integer $i\in\mathbb{Z}$.
\end{corollary}

The following lemma essentially resolves Case \ref{maxoutalgo:I} of Proposition \ref{prop:maxoutalgo}.

\begin{lemma}\label{lem:algo_case1}
There exists an algorithm with input a monomorphism $\psi: F(a, b)\rightarrow F(a, b)$ satisfying $\det(\Psi_{(a,b)})\neq\pm1$ and with output either the two maximal outer fixed points $[x]$ and $[x^{-1}]$ of $\psi$, or the trivial conjugacy class $[1]_F$ if $\psi$ has no non-trivial outer fixed points.
\end{lemma}

\begin{proof}
Let $x$ be a maximal outer fixed element. Abelianisation map considerations lead to the restrictions $(\sigma_a(x), \sigma_b(x))\Psi_{(a,b)}=(\sigma_a(x), \sigma_b(x))$, and $\gcd(\sigma_a(x), \sigma_b(x))=1$ as $x$ is primitive by Lemma \ref{lem:maxouter}. By replacing $x$ with $x^{-1}$ if necessary, we may additionally assume that $\sigma_a(x)\geq0$. The algorithm is as follows.

First compute the matrix $\Psi_{(a,b)}$ from the map $\psi$. Then find integers $p, q$ satisfying $(p, q)\Psi_{(a,b)}=(p, q)$, $\gcd(p, q)=1$ and $p\geq0$ (the above restrictions), or prove that no such pair exists: to find $(p, q)$, write the equation $(P, Q)\Psi_{(a,b)}=(P, Q)$ with variables $P, Q$ as the system of equations
\begin{equation}\label{equations}
\begin{cases}
P(\sigma_a(A)-1)+Q\sigma_a(B)=0\\
P\sigma_b(A)+Q(\sigma_b(B)-1)=0
\end{cases}
\end{equation}
where $A:=\psi(a)$ and $B:=\psi(b)$. In matrix form this is $(P,Q)(\Psi_{(a,b)}-I)=(0,0)$. Now, suppose $\det(\Psi_{(a,b)}-I)\neq 0$. Then $(P, Q)=(0, 0)$ is the unique solution to (\ref{equations}), and as $\gcd(0, 0)\neq1$ it follows that no pair satisfying the required conditions exists, and so $\psi$ has no non-trivial outer fixed points; hence, output from the algorithm the trivial conjugacy class $[1]_F$. Next, suppose $\det(\Psi_{(a,b)}-I) = 0$.
If $(\sigma_a(A), \sigma_a(B))\neq(1, 0)$ then let $p_0:=-\sigma_a(B)$ and $q_0:=\sigma_a(A)-1$, and note that $(p_0, q_0)\neq(0,0)$.
Else, let $p_0:=\sigma_b(B)-1$ and $q_0:=-\sigma_b(A)$, and note that $(p_0, q_0)\neq(0,0)$ since $\det(\Psi_{(a,b)})\neq\pm1$. Then $(P, Q)=(p_0, q_0)$ is a solution to (\ref{equations}).
To obtain the restrictions $P\geq0$ and $\gcd(P, Q)=1$: If $p_0<0$ then define $(p_1, q_1):=-(p_0, q_0)$, else define $(p_1, q_1):=(p_0, q_0)$. Finally define $(p, q):=\frac{1}{\gcd(p_1, q_1)}(p_1, q_1)$. Then $(P, Q)=(p, q)$ satisfies all the required conditions, and by Lemma \ref{lem:Matrices}, the pair $(p, q)$ is unique.

In the final step, using the description \cite{Cohen1981WhatDoes} of primitive elements in $F(a,b)$ (see Lemma \ref{lem:primitives}),  construct a primitive element $x$ of $F(a, b)$ such that $\sigma_a(x)=p$ and $\sigma_b(x)=q$. The conjugacy class $[x]$ is the unique conjugacy class of primitive elements which map to $(p, q)$ under the abelianisation map \cite[Corollary N4]{mks}.
Therefore, by Lemma \ref{lem:useAb}, either $\psi(x)$ is conjugate to $x$ (and so $x$ is a maximal outer fixed point) or $\psi$ has no outer fixed points. Therefore, check whether $\psi(x)$ is conjugate to $x$. If they are conjugate then output $[x]$ and $[x^{-1}]$ as the maximal outer fixed points, and if not, then output the trivial conjugacy class $[1]_F$.
\end{proof}

\section{Existence of non-trivial outer fixed points in powers of $\psi$}
\label{sec:ofp2}
We now consider Case \ref{maxoutalgo:II} of Proposition \ref{prop:maxoutalgo}, so non-surjective monomorphisms $\psi: F(a, b)\rightarrow F(a, b)$ with $\det(\Psi_{(a,b)})=\pm1$.
Under these restrictions, Proposition \ref{prop:KapMut} gives an algorithm which determines whether or not there exists some integer $k\neq0$ such that $\psi^k$ has a non-trivial outer fixed point. In Section \ref{sec:classify}, we apply this existence algorithm to find all the maximal outer fixed points of $\psi^k$ (although maximal fixed points are unique up to inversion, maximal outer fixed points may not be; see Example \ref{ex:4MOFP}).

At the end of this section, in Lemma \ref{lem:extract}, we explain why an algorithm studying the powers of $\psi$, rather than $\psi$ itself, is relevant.

\subsection{The hyperbolicity of the mapping torus}
Our algorithm to determine the existence of $k\neq0$ with $\psi^k$ having non-trivial outer fixed points first investigates the (lack of) hyperbolicity of the mapping torus $M_{\psi}$ of $\psi$:
\[
M_{\psi}=\langle a, b, t\mid a^t=\psi(a), b^t=\psi(b)\rangle.
\]
The hyperbolicity of $M_{\psi}$ is relevant by the following lemma. 
\begin{lemma}\label{lem:hypobstruction}\leavevmode
\begin{enumerate}[label=(\roman*)]
\item\label{list:hypobstruction1} If there exists an integer $k\neq0$ such that $\psi^k$ has an outer fixed point then $M_{\psi}$ contains a copy of $\mathbb{Z\times Z}$. 
\item\label{list:hypobstruction2} If $M_{\psi}$ is hyperbolic then $\psi^k$ has no outer fixed points for any $k\neq0$.
\end{enumerate}
\end{lemma}
\begin{proof}
For \ref{list:hypobstruction1}, if $\psi^k$ has an outer fixed point $w$, then $t^{-k}wt^k=w^g$ for some $g \in F(a,b)$, so $M_{\psi}$ contains $\langle t^{k}g^{-1}\rangle \times \langle w\rangle \cong \mathbb{Z\times Z}$. For \ref{list:hypobstruction2}, hyperbolic groups do not contain subgroups isomorphic to $\mathbb{Z}^2$ \cite[Corollary III.$\Gamma$.3.10]{Bridson1999metric}, so apply \ref{list:hypobstruction1}.
\end{proof}

We therefore wish to better understand those maps $\psi$ such that $M_{\psi}$ is not hyperbolic.
Note that if an identity of the form
\begin{equation}
\psi^p(x)=g^{-1}x^qg\label{eq:BSIdentity}
\end{equation}
holds in $F(a, b)$, where $p,q\in\mathbb{Z}\setminus\{0\}$, then $M_{\psi}$ is not hyperbolic: the equality corresponds to $t^{-p}xt^p=g^{-1}x^qg$ in $M_{\psi}$, and so (after writing $h=t^{-p}g$ and $y=g^{-1}xg$) the Baumslag--Solitar relation $hyh^{-1}=y^p$ holds in $M_{\psi}$; indeed the Baumslag--Solitar group $\operatorname{BS}(1, p)$ embeds into $M_{\psi}$ \cite[Lemma 2.3]{Kapovich2000Mapping}, which means that $M_{\psi}$ is not hyperbolic \cite{Gersten1991Rational}.
The following theorem, due to Mutanguha \cite[Corollary 7.4]{Mutanguha2018Hyperbolic}, says that if $M_{\psi}$ is non-hyperbolic then an identity of the form (\ref{eq:BSIdentity}) must hold. Hence, the existence of such identities is equivalent to the non-hyperbolicity of $M_{\psi}$.

\begin{theorem}[see {\cite[Corollary 7.4]{Mutanguha2018Hyperbolic}}]
\label{thm:HypMappingTori}
The group $M_{\psi}=\langle a, b, t\mid a^t=\psi(a), b^t=\psi(b)\rangle$ is hyperbolic if and only if $\psi^p(x) \nsim_{F(a, b)} x^q$ for all $x\in F(a,b) \setminus \{1\}$, and all $p,q\in\mathbb{Z}\setminus \{0\}$.
\end{theorem}

Similar to Kapovich's proof that hyperbolicity is decidable if $\psi$ is an ``immersion'' \cite[Corollary 5.6]{Kapovich2000Mapping}, this theorem can be used to produce a hyperbolicity tester for $M_{\psi}$. Indeed, the tester really outputs either a hyperbolicity constant $\delta$ for $M_{\psi}$ or an identity of the form (\ref{eq:BSIdentity}); we now investigate what happens in this second case.

\subsection{From hyperbolicity tester to outer fixed points}
Our immediate goal is Lemma \ref{lem:IdentityMatrixCase}, which links  identities of the form (\ref{eq:BSIdentity}) to outer fixed points.
The proof of this lemma is based on Magnus' method from the theory of one-relator groups.
Proposition \ref{prop:KapMut} then combines Lemma \ref{lem:IdentityMatrixCase} with the aforementioned hyperbolicity tester to give the algorithm we are after.
We first show that we may change the endomorphism $\psi$ and underlying basis, for the benefit of later proofs.

\begin{lemma}
\label{lem:ChangeBasis}
Let $\psi:F(a, b)\rightarrow F(a, b)$ be a non-surjective monomorphism such that the associated matrix $\Psi_{(a,b)}$ satisfies $\det(\Psi_{(a,b)}) = \pm1$.
Suppose that $\psi$ has a non-trivial maximal outer fixed point $\alpha$.

There exists a non-surjective monomorphism $\phi:F(a, b)\rightarrow F(a, b)$ and a basis $(x, y)$ for $F(a, b)$ such that the following hold:
\begin{enumerate}
\item\label{ChangeBasis:1} the associated matrix $\Phi_{(a,b)}$ satisfies $\det(\Phi_{(a,b)}) = \pm1$,
\item\label{ChangeBasis:2} $\psi^p(w)\sim w^q$ if and only if $\phi^p(w)\sim w^q$,
\item\label{ChangeBasis:3} $\phi(x)=x$,
\item\label{ChangeBasis:4} $\sigma_x(\phi(y))=0$,
\item\label{ChangeBasis:5} $\sigma_y(\phi^k(y))=\pm1$ for all integers $k \neq 0$.
\end{enumerate}
\end{lemma}

\begin{proof}
Consider $x\in\alpha$, so there exists an element $g\in F(a, b)$ such that $\psi(x)=g^{-1}xg$. Define $\phi:=\psi\gamma$ where $\gamma\in\inn(F(a, b))$ corresponds to conjugation by $g^{-1}$. Note that $\phi(x)=x$. Write $\Phi_{(a, b)}$ for the associated matrix of $\phi$ relative to the basis $(a, b)$. Note that $\phi$ remains a non-surjective monomorphism and $\det(\Phi_{(a, b)})= \pm1$ since $\det(\Psi_{(a,b)})= \pm1$, so (\ref{ChangeBasis:1}) holds. Furthermore, $\phi^p({w}) \sim \psi^p({w})$ for all $w\in F(a, b)$, so (\ref{ChangeBasis:2}) holds.

We now change the basis of $F(a, b)$ as follows. As the fixed point $x$ of $\phi$ is a primitive element of $F(a, b)$, by Lemma \ref{lem:maxouter}, there exists an element $z$ of $F(a, b)$ such that the pair $(x, z)$ forms a basis for $F(a, b)$. Let $y:=zx^{-n}$ where $n=\sigma_x(\phi(z))$. Then $(x, y)$ is a basis of $F(a, b)$, $\phi(x)=x$ and $\sigma_x(\phi(y))=0$, so (\ref{ChangeBasis:3}) and (\ref{ChangeBasis:4}) hold.
With respect to the basis $(x,y)$ the associated matrix is $\Phi_{(x, y)}:=\left(\begin{array}{cc}1 & 0\\ 0&\epsilon \end{array}\right)$, where $\epsilon=\pm1$ since $\det(\Phi_{(x, y)})=\det(\Phi_{(a, b)})=\pm1$. In particular, $\sigma_y(\phi^k(y))=\pm1$ for all $k \neq 0$, so (\ref{ChangeBasis:5}) holds.
\end{proof}

We now state Lemma \ref{lem:IdentityMatrixCase}, which is applied in the proof of Proposition \ref{prop:KapMut}.

\begin{lemma}
\label{lem:IdentityMatrixCase}
Let $\psi:F(a, b)\rightarrow F(a, b)$ be a non-surjective monomorphism such that the associated matrix $\Psi_{(a,b)}$ satisfies $\det(\Psi_{(a,b)}) = \pm1$.
Suppose that $\psi$ has a non-trivial maximal outer fixed point $\alpha$.

If $\psi^p({w}) \sim {w}^q$ for some ${w}\in F(a, b)\setminus\{1\}$ and $p, q\in\mathbb{Z}\setminus\{0\}$, then $q=\pm1$.
\end{lemma}

The proof of Lemma \ref{lem:IdentityMatrixCase} uses a trick borrowed from Magnus' method in the theory of one-relator groups. We refer the reader to McCool and Schupp's paper \cite{McCool1973HNN} for an account of the HNN-extension interpretation of Magnus' method: this is based on the observation that if $G=\langle a, b\mid R\rangle$ is a one-relator group with $\sigma_a(R)=0$ then we can view $G$ as an HNN-extension with stable letter $a$ by
writing $b_i:=a^{-i}ba^i$, and the word $R$ as a word $S$ over the letters $b_i$ (which is possible since $\sigma_a(R)=0$). If $m$ and $M$ are the minimum and maximum $i$, respectively, such that $b_i$ is in $S$, then
\begin{align}
G&\cong \langle a, b_m, \ldots, b_M\mid S(b_m, \ldots, b_M), b_m^a=b_{m+1}, \ldots, b_{M-1}^a=b_M\rangle.\label{align:HNNPres}
\end{align}
In the one-relator group $H=\langle b_m, \ldots, b_M\mid S(b_m, \ldots, b_M)\rangle$, by the Freiheitssatz \cite[Theorem 1]{McCool1973HNN}, the subgroups $B_m=\langle b_m, \ldots, b_{M-1}\rangle$ and $B_M=\langle b_{m+1}, \ldots, b_{M}\rangle$ are free on the given generators, and isomorphic via the map $b_m\mapsto b_{m+1}, \ldots, b_{M-1}\mapsto b_M$. Therefore, the presentation (\ref{align:HNNPres}) describes $G$ as an HNN-extension of the group $H$ with associated subgroups $B_m$ and $B_M$.

For example, if $G=\langle a, b\mid b^2a^{-1}b^2a^{-1}b^2a^2\rangle$ then
\[
G\cong\langle a, b_0, b_1, b_2\mid b_0^2b_1^2b_2^2, b_0^a=b_1, b_1^a=b_2\rangle
\]
is an HNN-extension with stable letter $a$ and base group $\langle b_0, b_1, b_2\mid b_0^2b_1^2b_2^2\rangle$.

\begin{remark}\label{rem:Magnus} In the proof of Lemma \ref{lem:IdentityMatrixCase} below, rather than using $\langle a, b\mid R\rangle$ with $\sigma_a(R)=0$, we have a presentation of the form
\[
M_{\phi}=\langle x, y, t\mid [x, t], t^{-1}y^{-1}tu(x, y)\rangle
\]
with $\sigma_x(u(x, y))=0$, and we wish $x$ to be the stable letter.
We also take the word $w=w(x, y)$ in the statement of Lemma \ref{lem:IdentityMatrixCase}, with the additional assumption that $\sigma_x(w(x, y))=0$, as input to this process.
The same idea works as in the one-relator case:
let $y_i:=x^{-i}yx^i$ and, since $\sigma_x(u(x, y))=0$ and $\sigma_x(w(x, y))=0$, write the words $u(x, y)$ and $w(x, y)$ as words $u'$ and $w'$, respectively, over the letters $y_i$. Let $m_u$ and $m_w$ be the minimum integers such that $y_{m_u}$ and $y_{m_w}$ are contained in $u'$ and $w'$ respectively, and $M_u$ and $M_w$ be the maximum such integers; let $m:=\min(m_u, m_w, 0)$ and $M=\max(M_u, M_w, 0)$, and include $y_{m}, \ldots, y_{M}$ as generators. Here $\min$ and $\max$ ensure that the word $w'$ and the relator $t^{-1}y^{-1}tu(x, y)$ can both be rewritten in terms of the generators $t$ and $y_i$ (we require $y_0$ for the relator $t^{-1}y^{-1}tu(x, y)$, even if it does not occur in the word $u'$). Then 
\begin{footnotesize}
\[
M_{\phi}\cong \langle x, t, y_{m}, \ldots, y_{M}\mid t^{-1}y_0^{-1}tu'(y_{m}, \ldots, y_{M}), t^x=t, y_{m}^x=y_{m+1}, \ldots, y_{M-1}^x=y_{M}\rangle.
\]
\end{footnotesize}
Therefore, by an analogous argument to the one-relator case, and again applying the Freiheitssatz, this presentation describes $M_{\phi}$ as an HNN-extension of the one-relator group $H=\langle t, y_{m}, \ldots, y_{M}\mid t^{-1}y_0^{-1}tu'(y_{m}, \ldots, y_{M})\rangle$ with associated subgroups $B_m=\langle t, y_{m}, \ldots, y_{M-1}\rangle$ and $B_M=\langle t, y_{m+1}, \ldots, y_{M}\rangle$.

For example, if $u(x, y)=y^2x^{-1}y^2x^{-1}y^2x^2$ then
\begin{align*}
M_{\phi}&=\langle x, y, t\mid [x, t], t^{-1}y^{-1}ty^2x^{-1}y^2x^{-1}y^2x^2\rangle\\
&\cong\langle x, t, y_0, y_1, y_2\mid t^{-1}y_0^{-1}ty_0^2y_1^2y_2^2, t^x=t, y_0^x=y_1, y_1^x=y_2\rangle
\end{align*}
 is an HNN-extension with stable letter $x$ and base group $\langle t, y_0, y_1, y_2\mid t^{-1}y_0^{-1}ty_0^2y_1^2y_2^2\rangle$.
\end{remark}

\begin{proof}[Proof of Lemma \ref{lem:IdentityMatrixCase}]
By (\ref{ChangeBasis:1}) and (\ref{ChangeBasis:2}) of Lemma \ref{lem:ChangeBasis}, it is sufficient to prove the result for the map $\phi$ from Lemma \ref{lem:ChangeBasis}. Let ${u}(x, y):=\phi(y)$ and consider the mapping torus
\[
M_{\phi}=\langle x, y, t\mid x^t=x, y^t={u}(x, y)\rangle.
\]
Now, $\sigma_x(u(x, y))=0$ by Lemma \ref{lem:ChangeBasis}(\ref{ChangeBasis:4}), and so $x$ has exponent-sum $0$ in both relators of $M_{\phi}$. Therefore, the exponent-sum map $\sigma_x: F(x, y, t)\rightarrow \mathbb{Z}$ induces an exponent-sum homomorphism ${\sigma}_x: M_{\phi}\rightarrow \mathbb{Z}$. By hypothesis $\phi^p({w}) \sim {w}^q$ in $F(x,y)$, so write $\phi^p({w})=g^{-1} {w}^q g$ with $g\in F(x,y)$; furthermore, ${w} \sim {w}^q$ in $M_{\phi}$ as $t^{-p}{w}t^p=\phi^p({w})$. Hence, ${\sigma}_x(w^q)={\sigma}_x(w)$. As ${\sigma}_x(w^q)=q{\sigma}_x(w)$ then either $q=1$, as required, or ${\sigma}_x(w)=0$. So assume that ${\sigma}_x(w)=0$.

By Remark \ref{rem:Magnus}, rewrite $M_{\phi}$ as an HNN-extension with stable letter $x$ and base group $H=\langle t, y_{m}, \ldots, y_{M}\mid t^{-1}y_0^{-1}tu'(y_{m}, \ldots, y_{M})\rangle$. As $\sigma_x({w})=0$ we can write $w(x, y)=w'(x^{-m_w}yx^{m_w}, ..., x^{-M_w}yx^{M_w})$ for $w', m_w, M_w$ as in Remark \ref{rem:Magnus}, so $w$ (in the HNN-group) can be viewed as ${w}'$ (in the base group $H$). Moreover, $t^{-p}{w}'t^p=g^{-1}({w}')^qg$ holds in the base group $H$. 

The base group $H$ is itself an HNN-extension with stable letter $t$, and we can apply Britton's lemma to the identity $t^{-p}{w}'t^p=g^{-1}({w}')^qg$ to get ${w}'\in \langle y_0\rangle\cup\langle u'(y_{m}, \ldots, y_{M})\rangle$. Hence, ${w}\in \langle y\rangle\cup\langle u(x, y)\rangle$ in $F(x,y)$.
Therefore, there exists some $k\in\mathbb{Z}\setminus\{0\}$ such that $\phi^p(y^k)=g^{-1}y^{qk}g$ or $\phi^{p+1}(y^k)=\phi(y^{qk})$. Since $\phi$ is injective the second identity implies that $\phi^p(y^k)=y^{qk}$, so it also implies the first identity. As roots are unique in free groups, $\phi^p(y^k)=y^{qk}$ gives $\phi^p(y)=y^q$. Thus ${\sigma}_y(\phi^p(y))={\sigma}_y(y^q)=q$, and by Lemma \ref{lem:ChangeBasis} (\ref{ChangeBasis:5}), we also have $\sigma_y(\phi^p(y))=\pm1$, so $q=\pm1$ as required.
\end{proof}

We now combine Lemma \ref{lem:IdentityMatrixCase} with the observations on the hyperbolicity of $M_{\psi}$.

\begin{proposition}\label{prop:KapMut}
There is an algorithm which determines whether or not there exists an integer $k\in\mathbb{Z}\setminus\{0\}$ such that $\psi^{k}$ has an outer fixed point, and, if one exists, outputs such an integer $k$ and an outer fixed point of $\psi^k$.
\end{proposition}

\begin{proof}
As described in the proof of \cite[Corollary 5.6]{Kapovich2000Mapping}, Theorem \ref{thm:HypMappingTori} provides a hyperbolicity tester for $M_{\psi}$: run in parallel an algorithm to find a hyperbolicity constant $\delta$ for $M_{\psi}$ and an algorithm to find identities of the form $\psi^p(x)=g^{-1}x^qg$ in $F(a, b)$, where $p,q\in\mathbb{Z}\setminus\{0\}$. This terminates, by Theorem \ref{thm:HypMappingTori}.

If the hyperbolicity tester outputs a hyperbolicity constant $\delta$, then $M_{\psi}$ is hyperbolic and so, by Lemma \ref{lem:hypobstruction}, there is no $k$ such that $\psi^{k}$ has an outer fixed point.

If the hyperbolicity tester outputs an identity $\psi^p(x)=g^{-1}x^qg$ with $q\neq\pm1$ then, by Lemma \ref{lem:IdentityMatrixCase}, there is no $k$ such that $\psi^{k}$ has an outer fixed point.

Finally, if the hyperbolicity tester outputs an identity $\psi^p(x)=g^{-1}x^qg$ with $q = \pm 1$, then the map $\psi^{2p}$ has $[x]$ as an outer fixed point, and so we output $k=2p$ as our integer and $[x]$ as the outer fixed point.
\end{proof}

We now connect the outer fixed points of $\psi^k$ to those of $\psi$.
The link is based on Theorem \ref{thm:OuterFPClassification} (the main result of Section \ref{sec:classify}), which states that the set $\MOFix(\psi)$ is finite.
We store an outer fixed point $\alpha$ as an outer fixed element $x\in\alpha$, so if an outer fixed point $\alpha$ is given we implicitly have a concrete $x\in\alpha$.

\begin{lemma} \label{lem:extract}
There is an algorithm with input the maximal outer fixed points $\MOFix(\psi^k)$ of $\psi^k$, $k\neq0$ arbitrary, and with output $\MOFix(\psi)$.
\end{lemma}

\begin{proof}
We obtain $\MOFix(\psi)$ from $\MOFix(\psi^k)$ as follows: for each $\alpha$ in $\MOFix(\psi^k)$, obtain a representative $x\in\alpha$, and if $\psi(x)\sim x$ then place $\alpha$ in $\MOFix(\psi)$.

This procedure terminates as $\MOFix(\psi^k)$ is finite by Theorem \ref{thm:OuterFPClassification}, and it provides all of $\MOFix(\psi)$ as clearly $\MOFix(\psi)\subseteq\MOFix(\psi^k)$.
\end{proof}

\section{Proof for Case \ref{maxoutalgo:II} of Proposition \ref{prop:maxoutalgo}}
\label{sec:classify}
In this section we classify, in Theorem \ref{thm:OuterFPClassification}, the maximal outer fixed points of the endomorphisms from Case \ref{maxoutalgo:II} of Proposition \ref{prop:maxoutalgo}, so non-surjective monomorphisms $\psi: F(a, b)\rightarrow F(a, b)$ with $\det(\Psi_{(a,b)})=\pm1$.
We then combine Theorem \ref{thm:OuterFPClassification} with Lemma \ref{lem:IdentityMatrixCase} to provide an algorithm which computes the maximal outer fixed points of such endomorphisms.
The section ends by proving Theorem \ref{thm:maxoutalgo}.

The proof of Theorem \ref{thm:OuterFPClassification} requires the technical Lemmas \ref{lem:formofB}, \ref{lem:form1} and \ref{lem:form2}, in which we assume $\psi(a)=a$ and aim to determine $\psi(b)$, given the existence of an additional maximal outer fixed point $\neq[a^{\pm1}]$ of $\psi$. Recall from Lemma \ref{lem:maxouter} that a maximal outer fixed element $y$ is primitive in $F(a,b)$, and so $y$ is as in Lemma \ref{lem:primitives}.

\begin{lemma}[Cohen, Metzler and Zimmerman \cite{Cohen1981WhatDoes}]\label{lem:primitives}
A primitive element in $F(a,b)$ is either equal to $a$ or $b$, or is a non-proper power of the form
\begin{enumerate}
\item\label{primitives:1} $a b^{\epsilon n_1}\cdots a b^{\epsilon n_k}$, or
\item\label{primitives:2} $b a^{\epsilon n_1}\cdots b a^{\epsilon n_k}$,
\end{enumerate}
\noindent up to conjugation and inversion, where $\epsilon=\pm1$, $k\geq 1$, and $n_i\in\{m, m+1\}$ for some $m\geq1$ for all $1 \leq i \leq k$.
\end{lemma}
Given a set $X$ together the set of its formal inverses $X^{-1}$, and given two words $A, B\in (X\cup X^{-1})^*$, we write $A=B$ when $A$ and $B$ represent the same element of $F(X)$, and we write $A\equiv B$ when $A$ and $B$ are precisely the same word. We write $U=A\circ B$ to mean that $U\equiv AB$ and no free reduction is possible between $A$ and $B$. We often abbreviate ``freely reduced'' to ``reduced''.

A \emph{syllable} is any maximal single generator subword in a word on $\{a^{\pm 1}, b^{\pm 1}\}$. An $a$-syllable is a syllable of the form $a^n$ or $a^{-n}$ where $n>0$ (so we do not distinguish between positive and negative powers). If a word contains a syllable then that word is implicitly assumed to be non-empty. For example, the word $B_0$ in the lemma below is assumed to be non-empty, as are the words $W$ in Lemmas \ref{lem:form1}, \ref{lem:form2}, and \ref{lem:MOFP}.

\begin{lemma}
\label{lem:formofB}
Let $\phi:F(a, b)\rightarrow F(a, b)$ be a non-surjective monomorphism such that $\phi(a)=a$ and $\det(\Phi_{(a,b)})=\pm1$.
Suppose $\phi(y) \sim y$ for some word $y\neq a^{\pm1}$ which is:
\begin{enumerate}
\item not a proper power, and
\item of the form $a^{\delta_a}b^{\delta_b n_1}\cdots a^{\delta_a}b^{\delta_b n_k}$ or $b^{\delta_b} a^{\delta_a n_1}\cdots b^{\delta_b} a^{\delta_a n_k}$, where $\delta_a, \delta_b=\pm1$, $k\geq 1$, $n_i\in\{m, m+1\}$ with $m\geq1$.
\end{enumerate}

If $\phi(b)^{\delta_b}$ starts with a $b$-syllable, with $\delta_b$ as in (2),
then $\phi(b)^{\delta_b}$ has at least $4$ syllables and ends with an $a$-syllable. 
That is, $\phi(b)^{\delta_b}\equiv b^rB_0b^sa^t$ with $r, s, t\neq0$ and $B_0$ a reduced word starting and ending in $a$-syllables.
\end{lemma}

\begin{proof}
Write $B:=\phi(b)^{\delta_b}$, with $B$ reduced, and note that by hypothesis $B$ starts with a $b$-syllable and $\sigma_b(B)=\pm1$ as $\det(\Phi_{(a,b)})=\pm1$. 

Suppose first that $B$ ends in a $b$-syllable and write $B=U^{-1} \circ B_1 \circ U$ for some reduced word $U$ of maximal length (so $B_1$ is cyclically reduced). Note that either (i) $U$ is non-empty and ends with a $b$-syllable, or (ii) $U$ is empty and $B$ is cyclically reduced. Note that $|B|>1$ as $\phi$ is non-surjective. Denote by $Y$ the free reduction of the word $\phi(y)$.
Then each of $y$ and $B$ can have two forms, and so we have four possibilities:
\begin{align*}
\text{(i)}~
Y&\equiv a^{\delta_a}U^{-1}B_1^{n_1}U\cdots a^{\delta_a}U^{-1}B_1^{n_k}U
&
\text{(ii)}~
Y&\equiv a^{\delta_a}B^{n_1}\cdots a^{\delta_a}B^{n_k}
\\
Y&\equiv U^{-1}B_1Ua^{\delta_an_1}\cdots U^{-1}B_1Ua^{\delta_an_k}
&
Y&\equiv Ba^{\delta_a n_1}\cdots Ba^{\delta_a n_k}
\end{align*}

In each case, the word $Y$ is cyclically reduced with $|Y|>|y|$ (as $|B|>1$). This is a contradiction as $\phi(y)=Y\sim y$. Hence $B$ ends in an $a$-syllable.

If $B$ has two syllables, then $B\equiv b^{r}a^t$ for some $r\in\mathbb{Z}$. As $\sigma_b(B)=\pm1$ we have $r=\pm1$, and hence $\phi\in\aut(F(a, b))$, which is a contradiction. Finally, $B$ cannot have three syllables, since it starts with a $b$-syllable and ends in an $a$-syllable. So $B$ has the required form.
\end{proof}


Suppose that $\psi$ fixes $a$ and there is a second, up to inversion, maximal outer fixed point $\beta$. Lemma \ref{lem:primitives} says that there is an element $y\in\beta^{\pm1}$ which has one of two forms. We first suppose $y$ is of the form (\ref{primitives:1}) from Lemma \ref{lem:primitives}.

\begin{lemma}
\label{lem:form1}
Let $\psi:F(a, b)\rightarrow F(a, b)$ be a non-surjective monomorphism such that $\psi(a)=a$ and $\det(\Psi_{(a,b)})=\pm1$.
If $\psi(y) \sim y$ for $y$ non-empty, not $a^{\pm1}$, not a proper power, and of the form
\[
y=ab^{\epsilon n_1}\cdots ab^{\epsilon n_k}
\]
with $\epsilon=\pm1$, $k\geq 1$, and $n_i\in\{m, m+1\}$ for some $m\geq1$,
then $y=ab^{\epsilon}$
and $\psi(b)^{\epsilon}\equiv a^{-j}W^{-1}a^{p}b^{\epsilon}a^qWa^{j-1}$ for some integers $p, q, j\in\mathbb{Z}$ with $p+q=1$ and for some reduced word $W\in F(a, b)$ starting and ending in $b$-syllables.
\end{lemma}

\begin{proof}
Let $y=ab^{\epsilon n_1}\cdots ab^{\epsilon n_k}$ be as in the statement of the lemma.
Now, $\psi(b)^{\epsilon}$ contains a $b$-syllable as $\psi$ is injective, so write $\psi(b)^{\epsilon}=a^{-j}Ba^j$ where $B$ is a reduced word starting with a $b$-syllable.
Let $\gamma\in\inn(F(a, b))$ correspond to conjugaction by $a^{-j}$. The result holds for $\psi$ if and only if the result holds for $\phi:=\psi\gamma$; we therefore consider $\phi$. Note that $\phi(b)^{\epsilon}=B$.
By taking $\delta_a:=1$ and $\delta_b:=\epsilon$ in Lemma \ref{lem:formofB}, we can write $B\equiv b^rB_0b^sa^t$ with $r, s, t\neq0$ as in Lemma \ref{lem:formofB}.
 
Now $\phi(y)=aB^{n_1}aB^{n_2}\cdots aB^{n_k}$. Since $|B|>1$, free reduction must happen within $B^{n_i}aB^{n_{i+1\pmod k}}$ for some $i$, as otherwise $\phi(y)$ is freely and cyclically reduced with $|\phi(y)|>|y|$, which contradicts $\phi(y) \sim y$ (this still holds when $k=1$, when considering $B^{n_1}aB^{n_1}$). Write $b^rB_0b^s\equiv V^{-1}\circ B_1\circ V$ for some (possibly empty) reduced word $V$ of maximal length, so $B\equiv V^{-1}B_1Va^t$.

Suppose $t\neq-1$, and write $C_i$ for the free reduction of $B^{n_i}a$. Then $|C_i|\geq4n_i$, and no free reduction happens between the $C_i$, so $U:=a^{-1}\phi(y)a$ has the form:
\[
U=B^{n_1}aB^{n_2}a\cdots B^{n_k}a\equiv C_1 \circ \cdots \circ C_n
\]
Therefore, $U$ is cyclically reduced, conjugate to $y$, and has length:
\[
|U|\geq 4\sum_{i=1}^k n_i\gneq k+\sum_{i=1}^k n_i=|y|
\]
Hence, $U\not\sim y$ a contradiction.

So $t=-1$. Then $U=V^{-1}B_1^{n_1}\cdots B_1^{n_k}V=V^{-1}B_1^nV$, where $n={\sum_{i=1}^kn_i}$. Hence, $\phi(y)$ and $y$ are conjugate to the freely and cyclically reduced word $B_1^n$. As $y$ is not a proper power, we must have $\sum_{i=1}^kn_i=\pm1$. Furthermore, since $n_i\geq1$ for all $i$, we have that $y=ab^{\epsilon}$, as required.

To obtain $\psi(b)^{\epsilon}$ we first continue to consider $\phi(b)^{\epsilon}$. Write $V = a^{p_0}W$ where $|p_0|$ is maximal, and so if $W$ is non-empty then it begins with a $b$-syllable. Indeed, if $W$ is non-empty then it also ends with a $b$-syllable as $B\equiv W^{-1}a^{-p_0}B_1a^{p_0}Wa^{-1}$ begins with a $b$-syllable.
Now, as
\[
\phi(y)=a\cdot W^{-1}a^{-p_0}B_1a^{p_0}Wa^{-1}\sim B_1
\]
is conjugate to $ab^{\epsilon}$, we have that either $B_1=ab^{\epsilon}$ or $B_1=b^{\epsilon}a$. If $B_1=ab^{\epsilon}$ then set $p:=-p_0+1$ and $q:=p_0$. Else, set $p:=-p_0$ and $q:=1+p_0$. In both cases we have that $\phi(b)^{\epsilon}=W^{-1}a^{p}b^{\epsilon}a^{q}Wa^{-1}$ where $p+q=1$. Finally, the word $W$ cannot be empty as $\phi$ is not an automorphism. Hence, from the definition of $\phi$ as $\psi\gamma$ for $\gamma\in\inn(F(a, b))$ conjugation by $a^j$, we have that $\psi(b)^{\epsilon}=a^{-j}\phi(b)^{\epsilon}a^j=a^{-j}W^{-1}a^{p}b^{\epsilon}a^{q}Wa^{j-1}$ as required.
\end{proof}


We now suppose that the element $y\in\beta^{\pm1}$ is of the form (\ref{primitives:2}) from Lemma \ref{lem:primitives}.

\begin{lemma}
\label{lem:form2}
Let $\psi:F(a, b)\rightarrow F(a, b)$ be a non-surjective monomorphism such that $\psi(a)=a$ and $\det(\Psi_{(a,b)})=\pm1$.
If $\psi(y) \sim y$ for $y$ non-empty, not $a^{\pm1}$, not a proper power, and of the form
$$y= ba^{\epsilon n_1}\cdots ba^{\epsilon n_k}$$ with $\epsilon=\pm1$, $k\geq 1$ and $n_i\in\{m, m+1\}$ for some $m\geq1$,
then $y=ba^{\epsilon n_1}$ and $\psi(b)\equiv a^{-j}W^{-1}a^{p}ba^qWa^{j-\epsilon n_1}$ for some integers $p, q, j\in\mathbb{Z}$ with $p+q=\epsilon n_1$ and for some reduced word $W\in F(a, b)$ starting and ending in $b$-syllables.
\end{lemma}

\begin{proof}
Since $\psi$ is injective, $\psi(b)$ contains a $b$-syllable, so write $\psi(b)=a^{-j}Ba^j$ where $B$ is a reduced word starting with a $b$-syllable and $j \in \mathbb{Z}$.
Let $\gamma\in\inn(F(a, b))$ correspond to conjugation by $a^{-j}$. The result holds for $\psi$ if and only if the result holds for $\phi:=\psi\gamma$; we therefore consider $\phi$. Note that $\phi(b)=B$.
By taking $\delta_a:=\epsilon$ and $\delta_b:=1$ in Lemma \ref{lem:formofB}, we can write $B\equiv b^rB_0b^sa^t$ with $r, s, t\neq0$ as in Lemma \ref{lem:formofB}, and furthermore write $b^rB_0b^s\equiv V^{-1}\circ B_1\circ V$ for some (possibly empty) reduced word $V$ of maximal length, so $B\equiv V^{-1}B_1Va^t$.

Let $y=ba^{\epsilon n_1}\cdots ba^{\epsilon n_k}$ be as in the statement of the lemma. We show now that $y=ba^{\epsilon n_1}$, that is, $k=1$. Assume $k>1$. Then:
\begin{align}
\phi(y)=V^{-1}B_1Va^{t+\epsilon n_1}V^{-1}B_1Va^{t+ \epsilon n_2}\cdots V^{-1}B_1Va^{t+ \epsilon n_k}.\label{eqn:1}
\end{align}
If $t\not\in\{-\epsilon m, -\epsilon (m+1)\}$ then the above word is both freely and cyclically reduced, and longer than $y$, which contradicts $\phi(y) \sim y$. Therefore, $t\in\{-\epsilon m, -\epsilon (m+1)\}$, and writing $Y$ for the freely reduced word representing $\phi(y)$ gives:
\[
Y\equiv V^{-1}B_1^{r_1}Va^{\epsilon_0}V^{-1}B_1^{r_2}Va^{\epsilon_0}\cdots a^{\epsilon_0}V^{-1}B_1^{r_i}Va^{\delta},
\]
where $\epsilon_0=\pm1$, $\delta\in\{0, \epsilon_0\}$, and $r_j>0$ for each $j\in\{1, \ldots, i\}$. Write $\overline{Y}$ for the cyclic reduction of $Y$; then $Y\neq \overline{Y}$ only if $\delta=0$ and $V$ is non-empty, whence $\overline{Y}=VYV^{-1}$.
If $i=1$ and $\delta=0$ then, as $y$ is not a proper power, we have $\phi(y)=Ba^{l}$ for some integer $l$. As $\phi$ is injective, $y=ba^l$ and so $k=1$ a contradiction.
If $i=1$ and $\delta=\epsilon_0$ or if $i>0$ then we claim that the cyclically reduced word $\overline{Y}$ contains an $a^{\pm1}$ which is a maximal $a$-syllable (so not part of an $a^{\pm2}$), and if $V$ is non-empty then $\overline{Y}$ contains both $b$-terms and $b^{-1}$-terms:
If $i=1$ and $\delta=\epsilon_0$ then $\overline{Y}=Y$, and the claim holds as $b^rB_0b^s\equiv V^{-1}\circ B_1\circ V$ begins and ends in $b$-terms.
If $i>1$ then the word $\overline{Y}$ still contains $U:=B_1^{r_1}Va^{\epsilon_0}V^{-1}B_1^{r_2}$ as a subword, and the claim follows as $B_1V$ ends in a $b$-term.
Since $y$ and $\overline{Y}$ are equal up to cyclic permutation, and as $\overline{Y}$ contains an $a^{\pm1}$ which is a maximal $a$-syllable, we have $m=1$.
Similarly, $V$ must be the empty word: $y$ and $\overline{Y}$ are equal up to cyclic permutation, and $y$ contains either $b$-terms or $b^{-1}$-terms (not both), while if $V$ is non-empty then $\overline{Y}$ contains both $b$-terms and $b^{-1}$-terms.
Therefore, $B_1\equiv b^rB_0b^s$ and so $|B_1|\geq3$. 

Now, $\overline{Y}$ satisfies (\ref{eqn:2}), below, since $i\geq1$, $\sum r_{\lambda}=k$ (follows from (\ref{eqn:1})), and because $3k\gneq |y|$ (holds as $m=1$ so $|ba^{\epsilon(m+1)}|=3$, and as primitives are not proper powers so $y\neq (ba^{\epsilon(m+1)})^k$, and so $|y|\lneq k|ba^{\epsilon(m+1)}|=3k$).
\begin{align}\label{eqn:2}
|\overline{Y}|&\geq 3\sum_{{\lambda}=1}^i{r_{\lambda}}+(i-1) =3k+(i-1) \geq 3k \gneq |y|.
\end{align}
Hence, $\overline{Y}$ and $y$ are not conjugate, and so neither are $\phi(y)$ and $y$, a contradiction. Therefore, $k=1$ and so $y=ba^{\epsilon n_1}$ as required.

To obtain $\psi(b)$ we first continue to consider $\phi(b)$.
Since $b^rB_0b^s\equiv V^{-1}B_1V$, $V$ cannot be an $a$-syllable. If $V$ is non-empty then it decomposes as $a^{m}W$, where $|m|$ is maximal and $W$ is non-empty and begins with a $b$-syllable. Indeed, $W$ also ends with a $b$-syllable as $B\equiv W^{-1}a^{-m}B_1a^{m}Wa^{t}$ begins with a $b$-syllable.
Now, the word
\[
\phi(ba^{\epsilon n_1})=W^{-1}a^{-m}B_1a^{m}Wa^{t+\epsilon n_1}\equiv b^rB_0b^sa^{t+\epsilon n_1}
\]
is reduced and contains at least two $b$-syllables.
Therefore, if ${t+\epsilon n_1}\neq0$ then $\phi(ba^{\epsilon n_1})$ is cyclically reduced and so each of its cyclic shifts contain at least two $b$-syllables. This is a contradiction as $ba^{\epsilon n_1}$ is a cyclic shift of $\phi(ba^{\epsilon n_1})$. Hence, $t=-\epsilon n_1$. Then as $ba^{\epsilon n_1}$ is a cyclic shift of $\phi(ba^{\epsilon n_1})=W^{-1}a^{-m}B_1a^{m}W$, and as $B_1$ is cyclically reduced we have that $B_1=a^{p_0}ba^{q_0}$ where $p_0+q_0=\epsilon n_1$. By taking $p:=p_0-m$ and $q:=q_0+m$, we see that $\phi(b)\equiv W^{-1}a^{p}ba^qWa^{-\epsilon n_1}$ for some integers $p, q\in\mathbb{Z}$ with $p+q=\epsilon n_1$.
Finally, $W$ cannot be empty as $\phi$ is not an automorphism.
Hence, from the definition of $\phi$ as $\psi\gamma$ for $\gamma\in\inn(F(a, b))$ conjugation by $a^j$, we have that $\psi(b)=a^{-j}\phi(b)a^j=a^{-j}W^{-1}a^{p}ba^{q}Wa^{j-\epsilon n_1}$ as required.
\end{proof}

We finally consider the case, not included in the previous two lemmas, when $y=b$. If $\psi(b)\sim b^{\pm1}$ then clearly there exists some $U\in F(a, b)$ such that $Ub^{\pm1}U^{-1}$; additionally, $U\neq a^{\pm k}$ is required because $\psi$ is not surjective.

\begin{lemma}\label{lem:y=b}
Let $\psi:F(a, b)\rightarrow F(a, b)$ be a non-surjective monomorphism such that $\psi(a)=a$ and $\det(\Psi_{(a,b)})=\pm1$. If $\psi(b)\sim b$, then $\psi(b)\equiv U^{-1}bU$, where $U\neq a^{\pm k}$, $k\in \mathbb{Z}$, is a reduced word ending in an $a$-syllable.
\end{lemma}

The previous three lemmas can be summarised in the following lemma.
\begin{lemma}\label{lem:MOFP}
Let $\psi:F(a, b)\rightarrow F(a, b)$ be a non-surjective monomorphism such that $\psi(a)=a$ and $\det(\Psi_{(a,b)})=\pm1$.
A maximal outer fixed point $\beta\neq[a^{\pm1}]$ exists if and only if 
 there exist some reduced word $W$ starting and ending in $b$-syllables, and $p, q, j, \epsilon, n\in\mathbb{Z}$, $\epsilon=\pm1$, $n\geq1$, such that either:
\begin{enumerate}[label=(\roman*)]
\item\label{MOFP:1} $p+q=1$ and $\psi(b)^{\epsilon}\equiv a^{-j}W^{-1}a^{p}b^{\epsilon}a^qWa^{j-1}$, or
\item\label{MOFP:2} $p+q=\epsilon n$ and $\psi(b)\equiv a^{-j}W^{-1}a^{p}ba^qWa^{j-\epsilon n}$,
\end{enumerate}
or there exists a word $U\neq a^{\pm k}$, $k\in \mathbb{Z}$, ending in an $a$-syllable such that
\begin{enumerate}[resume, label=(\roman*)]
\item\label{MOFP:3} $\psi(b)\equiv U^{-1}bU$.
\end{enumerate}
If such a $\beta\neq[a^{\pm1}]$ exists then either $ab^{\epsilon}\in\beta\cup\beta^{-1}$ in Case \ref{MOFP:1}, or $ba^{\epsilon n}\in\beta\cup\beta^{-1}$ in Case \ref{MOFP:2}, or $b\in\beta\cup\beta^{-1}$ in Case \ref{MOFP:3}.
\end{lemma}

\begin{proof}
Suppose $\beta\neq[a^{\pm1}]$ is a maximal outer fixed point. Then either $\beta = [b^{\pm 1}]$, in which case Lemma \ref{lem:y=b} applies, or otherwise by Lemma \ref{lem:maxouter}, $\beta$ is the conjugacy class of a primitive element and so there exists some $y\in\beta\cup\beta^{-1}$ of either form (\ref{primitives:1}) or (\ref{primitives:2}) from Lemma \ref{lem:primitives}. For form (\ref{primitives:1}), apply Lemma \ref{lem:form1} and Case \ref{MOFP:1} of the current lemma follows.
For form (\ref{primitives:2}), apply Lemma \ref{lem:form2} and Case \ref{MOFP:2} of the current lemma follows.

On the other hand, suppose Case \ref{MOFP:1} holds. Then
\[
\psi(ab^{\epsilon})=a\cdot a^{-j}W^{-1}a^{p}b^{\epsilon}a^qWa^{j-1}\sim a^{p+q}b^{\epsilon}=ab^{\epsilon},
\]
so $[ab^{\epsilon}]$ is a maximal outer fixed point such that $[ab^{\epsilon}]\neq[a^{\pm1}]$, as required.
Next, suppose Case \ref{MOFP:2} holds. Then
\[
\psi(ba^{\epsilon n})=a^{-j}W^{-1}a^{p}ba^qWa^{j-\epsilon n}\cdot a^{\epsilon n} \sim ba^{p+q}=ba^{\epsilon n},
\]
so $[ba^{\epsilon n}]$ is a maximal outer fixed point such that $[ba^{\epsilon n}]\neq[a^{\pm1}]$.
Finally, if Case \ref{MOFP:3} holds then $[b]$ is a maximal outer fixed point such that $[b]\neq[a^{\pm1}]$, as required.
\end{proof}

We are now ready for the main result of this section, Theorem \ref{thm:OuterFPClassification}. Part \ref{item:algoOFPxtoy} of Theorem \ref{thm:OuterFPClassification} is algorithmic. Recall that if we are given an endomorphism $\psi: F(a, b)\rightarrow F(a, b)$ then we are either given both $\psi(a)$ and $\psi(b)$ explicitly in the input, or these words can be pre-computed.

\begin{theorem}
\label{thm:OuterFPClassification}
Let $\psi:F(a, b)\rightarrow F(a, b)$ be a non-surjective monomorphism such that $\det(\Psi_{(a, b)})=\pm1$.
Then, up to inversion, there are at most two maximal outer fixed points of $\psi$.
Moreover, if $\alpha$ and $\beta$ are two maximal outer fixed points of $\psi$ with $\alpha\neq\beta^{\pm1}$ then the following hold:
\begin{enumerate}[label=(\roman*)]
\item\label{item:max2OFP1}
for every representative $x\in \alpha$ there exists a representative $y\in \beta$ such that $x$ and $y$ form a free basis of $F(a, b)$, 
\end{enumerate}
and the following algorithmic analogue also holds
\begin{enumerate}[label=(\roman*), resume]
\item\label{item:algoOFPxtoy}
given a representative $x\in\alpha$, a representative $y\in\beta$ can be computed such that $x$ and $y$ form a free basis of $F(a, b)$.
\end{enumerate}
\end{theorem}

\begin{proof}
Let $\alpha$ be a known maximal outer fixed point of $\psi$ as in the hypothesis, and let $x\in\alpha$ (when considering Part \ref{item:algoOFPxtoy}, $x$ is the element given to us).
Therefore, $\psi(x)=g^{-1}xg$ for some $g\in F(a,b)$ computable from $\psi$ and $x$. If $\gamma\in\inn(F(a,b))$ is conjugation by $g$, then $\phi:=\psi\gamma$ satisfies $\phi(x)=x$.
By Lemma \ref{lem:fixedprimitives}, $x$ is a primitive element of $F(a,b)$. We can then find another primitive $t$ such that $\{x,t\}$ forms a basis for $F(a,b)$
(such an element $t$ can be easily computed, see for example \cite{Cohen1981WhatDoes}); among the infinitely many possible $t$ just pick one.
We can explicitly determine $\phi(t)$ in terms of $x$ and $t$, since $t$ is a word on $a$ and $b$, and $\phi(a)$, $\phi(b)$ are known.
Note that if the theorem holds for $\phi$ then it also holds for $\psi$; in particular, $\phi$ and $\psi$ have identical outer fixed points.

If there exists some $\epsilon=\pm1$ such that $\phi(t)^{\epsilon}$ has the form in Lemma \ref{lem:MOFP} \ref{MOFP:1} then $y:=xt^{\epsilon}\in\beta\cup\beta^{-1}$, and by Lemma \ref{lem:form1} there are no other outer fixed points.
If there exists some $\epsilon=\pm1$ such that $\phi(t)$ has the form in Lemma \ref{lem:MOFP} \ref{MOFP:2} then $y:=tx^{\epsilon n}\in\beta\cup\beta^{-1}$, $n\in\mathbb{Z}\setminus\{0\}$, and by Lemma \ref{lem:form2} there are no other outer fixed points.
Otherwise, by Lemma \ref{lem:MOFP}, $\phi(t)$ has the form in Lemma \ref{lem:MOFP} \ref{MOFP:3}, and $y:=b\in\beta\cup\beta^{-1}$, and by Lemma \ref{lem:y=b} there are no other outer fixed points.
Therefore, there are at most two maximal outer fixed points of $\psi$.
Now, the two possible forms for $y$ imply that $\{x, y\}$ form a basis of $F(a,b)$ because $\langle x, y\rangle=\langle x, t\rangle$, and $\{x,t\}$ form a free basis of $F(a, b)$, so Part \ref{item:max2OFP1} of the theorem holds.

To prove Part \ref{item:algoOFPxtoy}, recall that $\phi(t)$ is known as a word over $x$ and $t$.
By assumption, the form of $\phi(t)$ corresponds to one of Cases \ref{MOFP:1}, \ref{MOFP:2}, or \ref{MOFP:3} of Lemma \ref{lem:MOFP}, and moreover the specific case can be identified and, in the appropriate cases, the integers $\epsilon$ and $n$ can be computed. This allows us to compute $y$, as required.
\end{proof}

In the proof of Theorem \ref{thm:OuterFPClassification} we choose the primitive $t$ from infinitely many elements. The specific choice of $t$ does not matter: Suppose we had instead chosen an element $t'$ so that $\{x,t'\}$ forms a basis, and obtained $y'$ as a maximal outer fixed point. As $\langle x, t\rangle=\langle x, t'\rangle$, we get $t'=x^it^{\lambda}x^k$ for $\lambda=\pm 1$ and some integers $i, k$ by analysing the form of primitives in $F_2$ given in Lemma \ref{lem:primitives}. Suppose (for simplicity) that $\lambda=\epsilon=1$ and $\phi(t)\equiv x^{-j}W^{-1}x^ptx^qWx^{j-n}$. A short calculation gives $\phi(t')=\phi(x^itx^j)\equiv x^{i-j+\square}W'^{-1}x^{p-i+\triangle}tx^{q-k-\triangle}W'x^{j-n+k-\square}$ as a word over $\{x,t'\}$, where $W'$ starts and ends in $t'$-syllables, and $\triangle, \square$ represent the powers of $x$ that occur at the beginning and end of $W$ when rewriting $W$ in terms of $x$ and $t'$. This gives $n'=p+q-i-k$ and $\epsilon'=1$, so $y'=(t')^{\epsilon'}x^{n'}=x^itx^kx^{n-i-k}=x^itx^{n-i}\sim y$. One can similarly verify that $[y]=[y']$ in the remaining cases.

\p{Examples of outer fixed points}
By Lemma \ref{lem:fpUnique}, if $\psi: F(a, b)\rightarrow F(a, b)$ is non-surjective then $\psi$ has at most one maximal fixed point (up to inversion). In contrast, Theorem \ref{thm:OuterFPClassification} allows for two maximal \emph{outer} fixed points (up to inversion). The following examples show that all cases of Theorem \ref{thm:OuterFPClassification} occur: there exist endomorphisms with one and with two maximal outer fixed points (up to inversion).

\begin{example}
\label{ex:4MOFP}
In the examples below, the maps $\psi_i$ are injective but not surjective and satisfy $\det(\Psi_{i,(a, b)})=\pm 1$.
\begin{enumerate}[label=\arabic*.]
\item The map $\psi_1: F(a, b)\rightarrow F(a, b)$ given by $\psi_1(a)=a$, $\psi_1(b)=baba^2$ has a single maximal outer fixed point $[a]$, up to inversion. This is because the word $baba^2$ does not match the possible images of $b$ given in Lemma \ref{lem:MOFP}, as these possible images each contain both $b$ and $b^{-1}$ (these images are $a^{-j}W^{-1}a^{p}b^{\epsilon}a^qWa^{j-1}$, $a^{-j}W^{-1}a^{p}ba^qWa^{j-\epsilon n}$, and $U^{-1}bU$, and $W$ and $U$ contain $b$-syllables).

\item The map $\psi_2: F(a, b)\rightarrow F(a, b)$ given by $\psi_2(a)=a$, $\psi_2(b)=a^{-2}b^{-1}aba^2ba^{-1}$ has maximal outer fixed points $[a]$ and $[ba^3]$, and inverses $[a^{-1}]$ and $[a^{-3}b^{-1}]$. 
 
As another example, we can view $\psi_2$ with the basis $x:=ab$ and $y:=b$ to obtain the map $\psi'_2: F(x, y)\rightarrow F(x, y)$ given by $\psi'_2(x)=yx^{-1}y^{-1}x^2y^{-1}xyx^{-1}$ and $\psi'_2(y)=yx^{-1}yx^{-1}y^{-1}x^2y^{-1}xyx^{-1}$. This map has maximal outer fixed points $[xy^{-1}]$ and $[y(xy^{-1})^3]$, and their inverses $[yx^{-1}]$ and $[(yx^{-1})^3y^{-1}]$.
\end{enumerate}
\end{example} 

\p{Computing {\boldmath$\MOFix(\psi)$}}
We now resolve Case \ref{maxoutalgo:II} of Proposition \ref{prop:maxoutalgo}.
Recall that $M_{\psi}:=\langle a, b, t\mid a^t=\psi(a), b^t=\psi(b)\rangle$, and that in order to describe all outer fixed points it is sufficient to find the maximal ones.
\begin{lemma}\label{lem:algo}
There is an algorithm with input a non-surjective monomorphism $\psi: F(a, b)\rightarrow F(a, b)$ with $\det(\Psi_{(a,b)})=\pm1$ and with output either the maximal outer fixed points of $\psi$, or the trivial conjugacy class $[1]_F$ if $\psi$ has no non-trivial outer fixed points.
\end{lemma}
\begin{proof}
The algorithm is as follows.
\begin{small}

\begin{itemize}

\item[\textbf{Input}:]  Images $\psi(a)$ and $\psi(b)$, where $\psi: F(a, b)\rightarrow F(a, b)$ with $\det(\Psi_{(a,b)})=\pm1$
\vspace{.1cm}

\item[\textbf{Step 1}:] \textbf{Determine existence of outer fixed points of powers of $\psi$}.
\vspace{.1cm}

\noindent Run in parallel the algorithm to find a hyperbolicity constant $\delta$ for $M_{\psi}$ and the algorithm which searches for an element $x\in F(a, b)\setminus \{1\}$ and integers $p, q\in\mathbb{Z}$ such that $\psi^p(x)\sim x^q$. By Theorem \ref{thm:HypMappingTori}, this algorithm terminates.\vspace{.1cm}
\begin{enumerate}[leftmargin=*]

\item[] If some $\delta$ is found then $\psi$ has no fixed points (outer or otherwise!) by Lemma \ref{lem:hypobstruction}, so output $[1]_F$.\vspace{.1cm}

\item[] Else, we found some $x\in F(a, b)\setminus \{1\}$ and $p, q\in\mathbb{Z}\setminus \{0\}$ such that $\psi^p(x)\sim x^q$.\vspace{.1cm}

\begin{itemize}[leftmargin=*]

\item[]If $q\neq\pm1$ then $\psi$ has no fixed points (outer or otherwise!) by Lemma \ref{lem:IdentityMatrixCase}, so output $[1]_F$. 

\item[]If $q=\pm1$ then $[x]$ is a maximal outer fixed point of $\psi^{2p}$, so feed $x$ and $p$ to Step 2. \vspace{.1cm}
\end{itemize}
\end{enumerate}
\item[\textbf{Step 2}:] \textbf{Find all maximal outer fixed points of $\psi^{2p}$, based on $x$ and $p$}.\vspace{.1cm}

\noindent Find $g$ such that $\psi^{2p}(x)=g^{-1}xg$. Compute the map $\phi:=\psi^{2p}\gamma$, where $\gamma\in\inn(F(a,b))$ is conjugation by $g^{-1}$, so that $\phi(x)=x$. Find $t$ such that $\{x,t\}$ form a basis of $F(a,b)$, and compute $\phi(t)$ in terms of $\{x,t\}$. \vspace{.1cm}
\begin{itemize}[leftmargin=*]
\item[]If $\phi(t)$ is not of any of the forms in Lemma \ref{lem:MOFP} \ref{MOFP:1}--\ref{MOFP:3}, applied to $F(x,t)$, then by Theorem \ref{thm:OuterFPClassification}, $\phi$, and so also $\psi^{2p}$, have the unique maximal outer fixed point (up to inversion) $[x]$, so we feed $[x]$ into Step 3.\vspace{.1cm}

\item[]If $\phi(t)$ is of one of the form in Lemma \ref{lem:MOFP} \ref{MOFP:1}--\ref{MOFP:3}, applied to $F(x,t)$, by Theorem \ref{thm:OuterFPClassification} we can compute a second maximal outer fixed point $[y]\neq [x^{\pm 1}]$ from $[x^{\pm1}]$, so we feed both $[x^{\pm1}]$ and $[y^{\pm1}]$ into Step 3. \vspace{.1cm}
\end{itemize}
Thus $\psi^{2p}$ has either $[x^{\pm 1}]$, or $[x^{\pm1}]$ and $[y^{\pm1}]$, as maximal outer fixed points.\vspace{.1cm}

\item[\textbf{Step 3}:] \textbf{Find all maximal outer fixed points of $\psi$}.\vspace{.1cm}

\noindent If $[z]$ is a maximal outer fixed point of $\psi$ then it is a maximal outer fixed point of $\psi^{2p}$. Hence, run through the maximal outer fixed points of $\psi^{2p}$ from Step 2 and verify whether they are maximal outer fixed points of $\psi$ itself. Output the results.
\end{itemize}
\end{small}
\end{proof}

We can now prove Theorem \ref{thm:maxoutalgo}, that is, find the outer maximal fixed points of a non-surjective monomorphism $\psi$.

\begin{proof}[Proof of Theorem \ref{thm:maxoutalgo}]
Recall the two cases of Theorem \ref{thm:maxoutalgo} stated in Proposition \ref{prop:maxoutalgo}.
Case \ref{maxoutalgo:I} of Proposition \ref{prop:maxoutalgo} follows from Lemma \ref{lem:algo_case1}.
Case \ref{maxoutalgo:II} of Proposition \ref{prop:maxoutalgo} follows from Lemma \ref{lem:algo}.
\end{proof}


\section{From outer fixed points to fixed points}
\label{sec:FixedPoints}
In this section we prove Theorem \ref{thm:basisexists}, that there exists an algorithm with input $\psi\in\emo(F_2)$ and output a basis for $\fix(\psi)$.
We split this section into three subsections. In Section \ref{sec:endoTwist} we use the endomorphism-twisted conjugacy problem for $F_2$ to link Theorems \ref{thm:basisexists} and \ref{thm:maxoutalgo}. In Section \ref{sec:solvingInstances} we prove that the endomorphism-twisted conjugacy problem for $F_2$ is soluble for certain instances. In Section \ref{sec:MainProof}, we use these instances and Section \ref{sec:endoTwist} to prove Theorem \ref{thm:basisexists}.

If in the future the endomorphism-twisted (or just monomorphism-twisted) conjugacy problem for $F_2$ is shown to be decidable in general, then Section \ref{sec:solvingInstances} may be disregarded; we structure the whole section so that the proofs remain clear when using the endomorphism-twisted conjugacy problem for $F_2$ as a black box.

\subsection{Connecting to the endomorphism-twisted conjugacy problem}
\label{sec:endoTwist}
For an endomorphism $\phi$ of the free group $F$, two elements $U, V\in F$ are \emph{$\phi$-twisted-conjugate}, written $U\sim_{\phi}V$, if there exists some $W\in F$ such that $U=\phi(W)VW^{-1}$; the corresponding decision problem is called the \emph{$\phi$-twisted-conjugacy problem for $F$}.
The relation $\sim_{\phi}$, also known as Reidemeister's relation, plays an important role in Nielsen fixed point theory and its study has become a fruitful research area
\cite{Felshtyn1994Reidemeister}
\cite{Jiang2005Primer}
\cite{Hart2005Algebraic}
\cite{Felshtyn2006TwistedBurnsideThm}
\cite{Felshtyn2007TwistedBurnsideFrobenius}
\cite{Goncalves2009Twisted}
\cite{Wong2010Combinatorial}
\cite{Yi2015Nielsen}
\cite{Jiang2018Some}
\cite{Araujo2020Twisted}.
The endomorphism-twisted conjugacy problem for free groups is known to be decidable for automorphisms \cite{Bogopolski2006Conjugacy} and for certain non-surjective maps \cite{Kim2016Twisted}, but is open for the specific endomorphisms which we require to link Theorems \ref{thm:basisexists} and \ref{thm:maxoutalgo}.

For a word $Z$, define the endomorphism $\varphi_Z: F(a, b)\rightarrow F(a, b)$ as
\[\varphi_Z(a)=a, \varphi_Z(b)= Z.\]
Lemma \ref{lem:WordstoFP} connects the existence of fixed points (in the conjugacy class of $a$) of the input map $\psi$ to the $\varphi_Z$-twisted conjugacy problem for words $P$ and $a^k$, where $P$ is given but the integer $k$ is unknown.
\begin{lemma}
\label{lem:WordstoFP}
Suppose $\psi \in \emo(F(a,b))$ satisfies $\psi(a)=P^{-1}aP$ and $\psi(b)=Q$, where $P, Q \in F(a,b)$ are given, and define $Z:=PQP^{-1}$.

There exist $W \in F(a,b)$ and $k\in\mathbb{Z}$ such that
\begin{equation}
P=\varphi_Z(W)a^kW^{-1}\label{mainEq}
\end{equation}
if and only if $[a]\cap\fix(\psi)\neq \emptyset$.
\end{lemma}

\begin{proof}
Suppose there exist $W \in F(a,b)$ and $k\in\mathbb{Z}$ such that (\ref{mainEq}) holds.
Then $W^{-1}aW$ is a fixed point as follows:
\begin{align*}
\psi(WaW^{-1})&=W(\psi(a), \psi(b))P^{-1}aPW^{-1}(\psi(a), \psi(b))\\
&=W(P^{-1}aP, Q)P^{-1}aPW^{-1}(P^{-1}aP, Q)\\
&=W(P^{-1}aP, P^{-1}ZP)P^{-1}aPW^{-1}(P^{-1}aP, P^{-1}ZP)\\
&=P^{-1}W(a, Z)aW^{-1}(a, Z)P\\
&=P^{-1}\varphi_Z(W)a\varphi_Z(W)^{-1}P\\
&=Wa^{-k}aa^kW^{-1}\\
&=WaW^{-1}.
\end{align*}
Next, suppose that $[a]\cap\fix(\psi)\neq \emptyset$, and let $x\in[a]\cap\fix(\psi)$. Let $W\in F(a, b)$ be such that $a=W^{-1}xW$. As above, $\psi(WaW^{-1})=P^{-1}\varphi_Z(W)a\varphi_Z(W)^{-1}P$. Now, $x=WaW^{-1}\in \fix(\psi)$, so $WaW^{-1}=P^{-1}\varphi_Z(W)a\varphi_Z(W)^{-1}P$. Therefore, $\varphi_Z(W)^{-1}PW$ centralises the generator $a$, and so there exists some $k\in\mathbb{Z}$ such that $\varphi_Z(W)^{-1}PW=a^k$, and the result follows.
\end{proof}

The map $\psi$ below has the behaviour described in Lemma \ref{lem:WordstoFP}.
\begin{example}
Consider the endomorphism $\psi$ given by $\psi(a)=b^{-1} a b$ and $\psi(b)=(a^2b)^{-1}ba^2b (a^2b)$. That is, $P=b$, $Q=(a^2b)^{-1}ba^2b (a^2b)$, and so $Z=a^{-2}ba^2ba^2$. In this case, $W=a^2ba$ and $k=-2$ satisfy the conditions in Lemma \ref{lem:WordstoFP}. A routine computation shows that $\psi(WaW^{-1})=WaW^{-1}$.
\end{example}

Lemma \ref{lem:WordstoFP} reduces the problem of algorithmically determining if $[a]\cap\fix(\psi)\neq \emptyset$ to algorithmically solving Equation (\ref{mainEq}) for $W \in F(a,b)$ and $k\in\mathbb{Z}$.
In Lemma \ref{lem:fromOuterToActualALGORITHM} we solve this equation, modulo our work on the $\varphi_Z$-twisted conjugacy problem in Section \ref{sec:solvingInstances}. First, we state a convention which we shall use.
\begin{remark}\label{Wbsyllable}
In the remainder of this section we usually assume that the word $W$ in (\ref{mainEq}) ends in a $b$-syallable, and in particular has the form
\begin{equation}
W(a, b)=a^{i_1}b^{j_1} \dots a^{i_n}b^{j_n}\label{formOfW},
\end{equation}
with all $i_m, j_m$ non-zero, except possibly $i_1$.
We may do this without loss of generality, as if $W$ ends in an $a$-syllable then this will cancel when forming $\varphi_Z(W)a^kW^{-1}$.
\end{remark}

\begin{lemma}
\label{lem:fromOuterToActualALGORITHM}
Let $Z\in F(a, b)$ be such that $\varphi_Z\in \emo(F(a, b))$ is injective but not surjective.
Assume that there exists an algorithm with input $P\in F(a, b)$ and $k\in\mathbb{Z}$ which determines whether or not $P$ and $a^k$ are $\varphi_Z$-twisted-conjugate.

Then there exists an algorithm which determines, on input a word $P\in F(a, b)$, whether or not there exist $W \in F(a,b)$ and $k\in\mathbb{Z}$ such that (\ref{mainEq}) holds.
\end{lemma}

\begin{proof}
We firstly prove that if $W \in F(a,b)$ and $k\in\mathbb{Z}$ are such that $W$ ends in a $b$-syllable and (\ref{mainEq}) holds then either $|k|\leq|Z|$ or $|W|\leq |P|$.
Suppose $|k|>|Z|$ and $|W|>|P|$, and consider (\ref{mainEq}) in the form $PW=\varphi_Z(W)a^k$.
As $\varphi_Z$ is injective, $Z$ is not a power of $a$, and can be written, cancellation-free, as $Z=a^{q_0}Z_0^{-1} Z_1Z_0a^{q_1}$ where $Z_0^{-1}Z_1Z_0$ begins and ends in $b$-syllables, $Z_1$ is cyclically reduced, and $Z_0$, $q_0$ and $q_1$ might be trivial; now observe (for example, via Stalling's foldings) that no cancellation happens between any $Z_1$ and any other word when forming $W(\varphi_Z(a), \varphi_Z(b))$. Then as $|k|>|Z|$ and $W$ ends in a $b$-syllable, $\varphi_Z(W)a^{k}$ must end in an $a$-syllable.
However, as $|W|>|P|$, $PW$ must end in a $b$-syllable, a contradiction.

Our algorithm is therefore as follows: Check, via the algorithm in the hypothesis, whether for any $|k|\leq|Z|$ a word $W$ ending in a $b$-syllable satisfying (\ref{mainEq}) exists.
If such a pair $(W, k)\in F(a, b)\times\mathbb{Z}$ exists, then output it.
Else, verify for each word $W$ of length $\leq|P|$ whether or not such a $k\in\mathbb{Z}$ exists using the generalised word problem for free groups.
If such a pair $(W, k)\in F(a, b)\times\mathbb{Z}$ exists, then output it.
Else, output ``no pair exists''.

The correctness of the algorithms follows from the fact that if $W$ satisfying (\ref{mainEq}) exists, then a word $W_0$ ending in a $b$-syllable exists (see Remark \ref{Wbsyllable}) such that $P=\varphi_Z(W_0)a^kW_0^{-1}$, and then $|k|\leq|Z|$ or $|W_0|\leq |P|$, as proven above.
\end{proof}

\subsection{On the $\mathbf{\varphi_Z}$-twisted conjugacy problem}
\label{sec:solvingInstances}

In this section we prove the existence of the algorithm from the assumptions of Lemma \ref{lem:fromOuterToActualALGORITHM}.
Recall that \[\varphi_Z(a)=a, \varphi_Z(b)= Z.\]

\begin{lemma}
\label{lem:twistedConjugacy}
Let $Z\in F(a, b)$ be such that $\varphi_Z\in \emo(F(a, b))$ is injective but not surjective.
There exists an algorithm with input $P\in F(a, b)$ and $k\in\mathbb{Z}$ which determines whether $P$ and $a^k$ are $\varphi_Z$-twisted-conjugate.
\end{lemma}

Our approach is as follows: We prove that there exists a computable bound $C$ on $|W|$, given in terms of the constants $P$, $k$ and $Z$, for any $W$ ending in a $b$-syllable such that (\ref{mainEq}) holds.
To do this, we first show the number of syllables in $W$ is bounded, and then that the lengths of syllables are bounded.
The brute-force algorithm of first computing the bound $C$ and then checking whether or not the equation holds for each $W\in F(a, b)$ of length $\leq C$ is our required algorithm. The restriction to those words $W$ ending in a $b$-syllable is sufficient, by Remark \ref{Wbsyllable}.

\p{The form of $\mathbf{Z}$}
As in the proof of Lemma
\ref{lem:fromOuterToActualALGORITHM}, if $\varphi_Z$ is injective then $Z$ is not a power of $a$, so can be written as $a^{q_0}Z_0^{-1}Z_1Z_0a^{q_1}$. 
 It is in fact sufficient to assume $q_0=0$: If $q_0\neq 0$ then decompose $Z$ as $a^{q_0}Z'$, where $Z'=Z_0^{-1}Z_1Z_0a^{q_1}$ starts with a $b$-syllable. Then $P=W(a, a^{q_0}Z')a^kW^{-1}$ if and only if $a^{-q_0}P=W(a, Z'a^{q_0})a^{k-q_0}W^{-1}$, and so $P$ and $a^k$ are $\varphi_Z$-twisted conjugate if and only if $a^{-q_0}P$ and $a^{k-q_0}$ are $\varphi_{Z'a^{q_0}}$-twisted conjugate, so we can consider this new problem instead. Therefore, as a freely reduced word, we shall assume $Z$ has the form
\begin{equation}
Z=Z_0^{-1}Z_1Z_0a^q\label{formOfZ},
\end{equation}
where $q\in \mathbb{Z}$ and $Z_0^{-1}Z_1Z_0$ begins and ends with $b$-syllables. Note that free reductions within $W(a,Z)$ will not be affect the $Z_1$'s.

\p{Notation}
By a \emph{long} syllable we mean a syllable of length $\geq 2$.
We will denote by $t_a(W)$ the total number of occurrences of $a$ or $a^{-1}$ in the freely reduced form of $W$, by $s_a(W)$ the total number of $a$-syllables, by $s_a^{(2)}(W)$ the total number of long $a$-syllables, and by $s(W)$ the total number of syllables.
For example, if $W=a^2 b a^{-1} b^3$ then $t_a(W)=3$, $s_a(W)=s_b(W)=2$, $s_a^{(2)}=s_b^{(2)}=1$, and $s(W)=4$.


\p{Bounding the number of syllables}
We start by bounding the number of syllables in a solution $W$ to the $\varphi_Z$-twisted-conjugacy problem.
We first deal with the case when $q=0$ in the form $Z=Z_0^{-1}Z_0Z_0a^q$, as in (\ref{formOfZ}). We then consider the cases $t_b(Z_1)>1$, $t_b(Z_1)=1$, and $t_b(Z_1)=0$, each of which requires different methods.


\begin{lemma}
\label{lem:Case1non-cycRed}
Suppose that $\varphi_Z$ is injective but not surjective, that $q=0$ in the form of $Z$, and $W(a,b)$ is as in (\ref{formOfW}).
If (\ref{mainEq}) holds then $s(W)\leq |P|+2$.
\end{lemma}

\begin{proof}
Using the form (\ref{formOfW}) of $W$, after free reduction we have
\[
W(a, Z)=a^{i_1}Z_0^{-1}Z_1^{j_1}Z_0 \dots a^{i_n}Z_0^{-1}Z_1^{j_n}Z_0.
\]
In particular, $s(W(a, Z))\geq s(W)$.
As $Z$ starts and ends with $b$-syllables, we have two cases: either $Z$ is a power of $b$, so $Z=b^r$, or $s(Z)\geq3$.

Suppose that $Z=b^r$. Then $|r|>1$ as $\varphi_Z$ is non-surjective. Now, the only way cancellation can occur when we form $W(a, Z)a^kW^{-1}$ is if $k=0$, and then cancellation may occur between the end $b$-syllable of $W(a, Z)$, which is $b^{rj_n}$ and first $b$-syllable of $W^{-1}$, which is $b^{-j_n}$. As $|r|>1$, these two syllables do not completely cancel and so not further cancellation can occur. Therefore, $2s(W)-1\leq |P|$ and the inequality follows.

Suppose that $s(Z)\geq3$. Then $s(Z^{j_m})=2s(Z_0)+|j_m|s(Z_1)$ with $|j_m|\geq1$, and overall we have the following, where the last line is as $s(Z)=2s(Z_0)+s(Z_1)\geq3$:
\begin{align*}
s(W(a, Z))
&=\sum_{m=1}^n|i_m|+\sum_{q=1}^n(2s(Z_0)+|j_m|s(Z_1))\\
&\geq n-1 + n (2s(Z_0)+s(Z_1))\\
&\geq 4n-1.
\end{align*}
Now, $s(a^{-k}W^{-1}(a, b))\leq 2n+1$ and so, as $s(UV)\geq |s(U)-s(V)|$ and recalling that $s(W(a, Z))\geq s(W)$, we have:
\begin{align*}
|P|&\geq s(P)\\
&=s(W(a, Z)a^{-k}W^{-1}(a, b))\\
&\geq
|s(W(a, Z))-s(a^{-k}W^{-1}(a, b))|\\
&=
s(W(a, Z))-s(a^{-k}W^{-1}(a, b))\\
&\geq
(4n-1)-(2n+1)=2n-2.
\end{align*}
As $2n\geq s(W)$, we therefore have $|P|+2\geq s(W)$ as required.
\end{proof}

We now resolve the case of $t_b(Z_1)>1$. Our result here is for arbitrary words $U$ and $V$ as input, rather than just $U=P$ and $V=a^k$.

\begin{lemma}
\label{lem:SylCase1}
Suppose that $\varphi_Z$ is injective but not surjective, $t_b(Z_1)>1$ in the form of $Z$, and $W(a,b)$ is as in (\ref{formOfW}).
If $U, V, W\in F(a, b)$ are such that $U=\varphi_Z(W)VW^{-1}$, then $s(W)\leq 2(|U|+|V|)+1$.
\end{lemma}

\begin{proof}
Rewrite $W$ as a word over $a$ and $ba^{-q}$, so $W(a, b)=W_0(a, ba^{-q})$, and note that $t_b(W)=t_b(W_0)$. The word $\varphi_Z(W_0(a, ba^{-q}))=W_0(a, Z_0^{-1}Z_1Z_0)$ is freely reduced as written, and so we have the following:
\begin{align*}
U&=\varphi_Z(W_0(a, ba^{-t}))VW^{-1}\\
UWV^{-1}&=W_0(a, Z_0^{-1}Z_1Z_0)  \\
t_b(UWV^{-1})&=t_b(W_0(a, Z_0^{-1}Z_1Z_0))\\
t_b(U)+t_b(W)+t_b(V)&\geq 2s_b(W_0)t_b(Z_0)+t_b(W_0)t_b(Z_1)\\
t_b(U)+t_b(V)&\geq t_b(W_0)(t_b(Z_1)-1)
\end{align*}
As $t_b(Z_1)>1$, we have $t_b(U)+t_b(V)\geq t_b(W_0)(t_b(Z_1)-1)\geq t_b(W_0)=t_b(W)$, so $2(|U|+|V|)\geq2(t_b(U)+t_b(V))\geq 2t_b(W)\geq 2s_b(W)\geq s(W)-1$, and the inequality $s(W)\leq 2(|U|+|V|)+1$ follows.
\end{proof}


Assuming $Z$ is as in (\ref{formOfZ}) and $W(a,b)$ is as in (\ref{formOfW}), we say that an $a$-syllable $a^{i_m}$, $m\geq 2$, of $W(a,b)$ is \emph{cancelling} in $W(a, Z)$ if either both $j_{m-1}$ and $j_m$ are positive and $i_m=-q$, or both $j_{m-1}$ and $j_m$ are negative and $i_m=q$. This means that within $Z^{j_{m-1}}a^{i_m}Z^{j_{m}}$, the entire $a^{i_m}$ cancels with the adjacent $a$-syllable of $Z^{j_{m-1}}$ or $Z^{j_{m}}$, which implies $Z_0$ and $Z_0^{-1}$ must cancel as well. We shall write $c$ be the number of cancelling $a$-syllables in the word $W(a,b)$.

We record now some identities needed later.
The word $Z_0^{-1}Z_1Z_0$ starts and ends with $b$-syllables and is freely reduced as written, so after free reduction $W(a, Z)$ may be viewed as a word over $a$, $Z_0$ and $Z_1$. Therefore, if $q\neq0$ in the form of $Z$ then:
\begin{equation}
s_{Z_1}(W(a,Z)) = t_b(W)\label{Z1SymmablesInWaZ}
\end{equation}
\begin{equation}
s_{Z_0}(W(a,Z)) = 2(t_b(W)-c)\label{Z0SymmablesInWaZ}
\end{equation}
Moreover, the $Z_0$ terms can never be adjacent and so we have:
\begin{equation}
s_{Z_0}^{(2)}(W(a, Z))=0\label{Z0LongSymmablesInWaZ}
\end{equation}
We usually apply identities (\ref{Z0SymmablesInWaZ}) and (\ref{Z0LongSymmablesInWaZ}) in tandem. For example, if $t_b(Z_1)=0$ then (\ref{Z0SymmablesInWaZ}) implies that $s_{b}(W(a,Z)) \leq 2s_b(Z_0)(t_b(W)-c)$, with the inequality coming from the fact that $s_b(Z_0^2)\leq2s_b(Z_0)$. By (\ref{Z0LongSymmablesInWaZ}), every pair of $Z_0$s is separated and so the inequality is in fact an equality, so here $s_{b}(W(a,Z)) = 2s_b(Z_0)(t_b(W)-c)$.


The following proofs also use some notation. We write $f(n) \sim g(n)$ if there exists a computable constant $C$ (computable from the words $P$ and $Z$ in the proofs) such that $|f(n) - g(n)|<C$, where $f(n)$ and $g(n)$ are non-negative-valued functions. This relation is transitive, so if $f(n) \sim g(n)$ and $g(n) \sim h(n)$ then
\begin{equation}
f(n) \sim h(n)\label{approx1}
\end{equation}
because $|f(n) - h(n)|\leq|f(n) - g(n)|+|g(n) - h(n)|$.
We can further apply these approximations, for example if $f(n)\sim g(n)$, $h(n)\sim k(n)$ then
$x f(n)+y h(n)+z\sim x g(n)+y k(n)+z$ 
for $x, y, z\in\mathbb{N}\cup\{0\}$, which follows from applying (\ref{approx1}) twice.

In the following we settle the case $t_b(Z_1)=1$, with $Z_1$ cyclically reduced.
\begin{lemma}
\label{lem:Case2Z1has-one-t}
Suppose that $\varphi_Z$ is injective but not surjective, $Z=Z_0^{-1}a^{p_0}b^{\epsilon}a^{p_1}Z_0a^q$ with $\epsilon=\pm1$, $p_0p_1 \geq 0$, and $W(a,b)$ is as in (\ref{formOfW}).
If (\ref{mainEq}) holds then there exists a constant $D$, computable from $P$ and $Z$, such that $s(W)\leq D$.
\end{lemma}

\begin{proof}

Let $\beta:=t_b(Z_0)\geq 1$, with the inequality holding as $\varphi_Z$ is non-surjective.

The proof consists of two parts. In the first part we resolve the case of $p_0=0=p_1$, and for the remaining cases we show that `most' $a$-syllables in $W(a,b)$ are cancelling, and in the second part we show that a positive proportion of the $a$-syllables in $W(a,b)$ are non-cancelling. The two parts together lead to the bound on $s(W)$, as follows. Let $c$ be the number of cancelling $a$-syllables. In the first part we prove $c \sim s_b(W)$ (and clearly $s_b(W)\geq s_a(W)\geq c$), while in the second part we will show $c<s_b(W)-s_b(W)/\alpha$, where $\alpha:=s_a(Z_0)$. These two facts together give $s_b(W)\sim 0$, and so $s(W)$ is bounded by a computable constant.

We proceed with the first part of the proof and view $W(a, Z)$ as a word over $a$, $Z_0$ and $Z_1$; we have that $t_b(W(a,Z))=(\beta-\lambda) s_{Z_0}(W(a,Z))+s_{Z_1}(W(a,Z))$, where $\lambda\in\{0, 1, 2\}$ records the possibility that the final syllable of $Z_0$ is a $b$-syllable which merges with the initial or terminal $b$-syllable of $Z_1^{j_m}$. Hence, by (\ref{Z1SymmablesInWaZ}), (\ref{Z0SymmablesInWaZ}) and (\ref{Z0LongSymmablesInWaZ}) we have that $t_b(W(a,Z))=2(\beta-\lambda)(t_b(W)-c)+t_b(W)$. Now, by (\ref{mainEq}) we have $|t_b(W)-t_b(W(a, Z))|\leq t_b(P)$, so $t_b(W)\sim t_b(W(a, Z))$, and then combining this with the first identity and simplifying (and using the fact that $\beta$ and $\lambda$ are computable constants) we have that $t_b(W)\sim c$. As $c\leq s_a(W)\leq s_b(W)\leq t_b(W)$, we further have $s_a(W)\sim c$ and $s_b(W)\sim c$. The approximations $t_b(W)\sim c \sim s_b(W)$ also imply that $|j_m|=1$ for all $1\leq m \leq n$, up to $|P|$ of them.

Assume that $p_0=0=p_1$. There are two ways an $a$-syllable of $W(a, Z)$ can occur: either as a non-cancelling $a$-syllable of $W$, or within a $Z_0$-syllable. There are $s_a(W)-c$ syllables of the first kind, and $2s_a(Z_0)(t_b(W)-c)$ of the second form. Together with $s_a(W)-c\sim0$ and $t_b(W)-c\sim0$ from the above paragraph, we have
$s_a(W(a, Z))=(s_a(W)-c)+2s_a(Z_0)(t_b(W)-c)\sim0$.
As (\ref{mainEq}) holds, $s_a(W)\sim s_a(W(a, Z))$, which we can combine with the above, via (\ref{approx1}), to get $s_a(W)\sim0$, so $s_a(W)$ is bounded by a computable constant, as required. Hence, this case is resolved.

Assume that $p_0\neq0$ or $p_1\neq0$. Since $s_a(W)\sim s_a(W(a, Z)) \sim c$, all (up to a constant number of) $a$-syllables in $W$ are cancelling; moreover, we have seen that $|j_m|=1$ for all $1\leq m \leq n$ (except at most $|P|$ of them), so all (up to a constant) $Z_0$-syllables will cancel out in $W(a,Z)$. Then most $a$-syllables in the reduced form of $W(a,Z)$ will have the form $a^{\pm (p_0+p_1)}$. On the other hand, most $a$-syllables in $W$ are cancelling, that is, of the form $a^{\pm q}$. To get (\ref{mainEq}), sufficiently many of the $a$-syllables in $W(a,b)$ and $W(a,Z)$ must match, so therefore $|p_0+p_1|=|q|$.

In the second part of the proof we show that 
among any $s_a(Z_0)+2$ consecutive $a$-syllables in $W(a,b)$, at least one is non-cancelling. We extend the non-cancelling notion to say that an $a$-syllable in $W(a,Z)$ which is part of a $Z$-syllable will have \emph{non-cancelling properties} if it is either $\neq a^{\pm q}$, or if equal to $a^{\pm q}$, the neighbouring $b$-syllables have different signs. The relatively frequent occurrence of non-cancelling $a$-syllables, which we prove below, is initiated by the fact that the last $Z$-syllable in $W(a,Z)$ contains some $a$-syllable with non-cancelling properties, and this syllable must match an $a$-syllable in $W^{-1}(a,b)$ in order to satisfy (\ref{mainEq}): that is, a non-cancelling $a$-syllable $a^{i_t}$ in $W^{-1}(a,b)$ must exist in a position $t$ that can be approximated. The non-cancelling $a$-syllable $a^{i_t}$ in $W^{-1}(a,b)$ will in turn show the existence of an $a$-syllable with non-cancelling properties in $W(a,Z)$ in a position that can be approximated, based on $t$ and $s_a(Z_0)$; by continuing this process we show there are (relatively) frequent non-cancelling $a$-syllables in $W(a,b)$.

Write $\alpha:=s_a(Z_0)$. We may assume $Z_0$ starts and ends with $b$-syllables, as otherwise we can change the $p_0$ and $p_1$ accordingly.
We start with the last $Z$-syllable of $W(a,Z)$, $Z^{j_n}$ (assume wlog $j_n>0$); this cancels in the RHS of (\ref{mainEq}), or otherwise $W^{-1}(a,b)$ will stay intact in the RHS of (\ref{mainEq}) and we easily get $|W|<|P|$. We also assume $k=-q$ as otherwise $W^{-1}(a,b)$ and $W(a,Z)$ do not cancel and the bound on $|W|$ follows. Since the suffix $a^{p_0}b^{\epsilon}a^{p_1}Z_0a^q$ of $Z^{j_n}$ cannot cancel within $W(a,Z)$ (because the $Z_1$'s are not affected by cancellations), it must cancel with $a^kW^{-1}(a,b)$. We assume first that $p_1\neq 0$ and claim that $a^{p_1}$ has non-cancelling properties: if $p_1\neq \pm q$ this is immediate, and if $p_1= \pm q$ then either $p_0=0$ or $p_0=\pm 2q$ since $|p_0+p_1|=|q|$. If $p_0=0$, this which forces the $b$-syllables to the left ($b^{\epsilon}$) and right ($Z_0$) of $a^{p_1}$ to have different signs in order for $Z$ as in the hypothesis to be freely reduced. If $p_0=\pm 2q$, then we can take it as the starting non-cancelling $a$-syllable to be used in the approach below instead of $p_1$. Similarly, if $p_1=0$, we can use $p_0$ as the non-cancelling $a$-syllable at the beginning of the process. The details in all cases follow the pattern below.

We start with $a^{p_1}$, which has non-cancelling properties and is the $\alpha+1$st $a$-syllable when counting the $a$-syllables in $W(a,Z)$ from right to left. As it must cancel with the $\alpha+1$st $a$-syllable in $W^{-1}(a,b)$, we get $a^{i_{n-\alpha}}=a^{p_1}$.
Bouncing back to $W(a,Z)$, let $n-f$ be the position corresponding to $a^{i_{n-\alpha}}$ among the roughly $n$ $a$-syllables in the reduced form of $W(a,Z)$; even if $a^{i_{n-\alpha}}$ gets multiplied with $a^{\pm q}$ from within $Z$, the resulting $a$-syllable is non-cancelling, and so $a^{i_{n-f}}$ must be non-cancelling in $W^{-1}(a,b)$. 
If we suppose all $a$-syllables after $a^{i_{n-\alpha}}$ are cancelling, the number of $a$-syllables in the reduced form of $a^{i_{n-\alpha+1}}Z^{j_{n-\alpha+1}} \dots a^{i_n}Z^{j_n}$ is $\geq 2\alpha+(\alpha+2)=3\alpha+2$, where $2\alpha$ represents the number of $a$-syllables in the $Z^{-1}_0$ and $Z_0$ (which are guaranteed to not cancel in the prefix of $Z^{j_{n-\alpha+1}}$ and suffix of $Z^{j_n}$), and $\alpha+2$ counts the remaining $a$-syllables from the $Z_1$s. So $f\geq 3\alpha+2$. Since any non-cancelling $a^{i_{m}}$ in $W(a,Z)$ is followed by a $Z^{-1}_0$ and a non-cancelling $a^{\pm p_i}$ (part of $Z_1$), we get that $a^{i_{n-f+\alpha}}=a^{i_{n-2\alpha-2}}$ is also non-cancelling. If $f=3\alpha+2$ then both $a^{i_{n-\alpha}}$ and $a^{i_{n-2\alpha-2}}$ are non-cancelling, and they are distance $\alpha+2$ apart, so this is the beginning of the behaviour stated in the claim. Then the syllable $a^{i_{n-2\alpha-2}}$ will imply the existence of a non-cancelling syllable $a^{i_{n-2\alpha-2- f'}}$, and this process continues. 
Now if $f>3\alpha+2$ then there is at least another non-cancelling syllable between $a^{i_n}$ and $a^{i_{n-\alpha}}$, say $a^{i_{n-\alpha'}}$, with $\alpha'>\alpha$ maximal, so we use the same process for $\alpha'$. For example, if $5\alpha>f>3\alpha+2$, then there is one non-cancelling $a$-syllable between $a^{i_n}$ and $a^{i_{n-\alpha}}$, so between $a^{i_n}$ and $a^{i_{n-\alpha}}$ there are at least $4$ occurrences of $Z_0^{\pm 1}$, each followed or preceded by non-cancelling $a$-syllables, and an easy computation shows that there is a non-cancelling $a^{i_{n-\alpha'}}$ with $\alpha'=2\alpha+\delta$, where $|\delta|\leq 2$. For higher values of $f$ we get more non-cancelling syllables between $a^{i_n}$ and $a^{i_{n-\alpha}}$, but the computations will show that $\alpha'=2\alpha+\delta$, where $|\delta|\leq 2$. Following this strategy of identifying non-cancelling syllables in $W(a,b)$ and their counterparts in $W(a,Z)$, we get the claim.

\end{proof}

Our last main case is when $Z_1$ in (\ref{formOfZ}) satisfies $t_b(Z_1)=0$. Here, $Z_1$ is a non-trivial power of $a$ so we shall write $Z=Z_0^{-1}a^rZ_0a^q$, $r\neq0$. We have an ambient assumption that $\varphi_Z$ is injective, so we additionally have that $s_b(Z_0)\geq1$.
There are two cases to consider here: when $Z_0$ contains a long $b$-syllable and when it does not. In both cases we compare the ``pattern'' of the $b$-syllables in $W(a, b)$ with those in $W(a, Z)$.


We now consider the case when $s_b^{(2)}(Z_0)\geq1$ in the form of $Z$, so when $Z_0$ contains a long $b$-syllable. In the following, the constant $E$ may be computed to be $|Z|\times|P|$.

\begin{lemma}
\label{lem:SylCase2a}
Suppose that $\varphi_Z$ is injective but not surjective, $Z=Z_0^{-1}a^rZ_0a^q$ for $r, q\neq0$, and $W(a,b)$ is as in (\ref{formOfW}).
Let $s_b^{(2)}(Z_0)\geq1$.
If (\ref{mainEq}) holds then there exists a constant $E$, computable from $P$ and $Z$, such that $s(W)\leq E$.
\end{lemma}

\begin{proof}
Write $\beta:=s_b(Z_0)$ and $\gamma:=s_b^{(2)}(Z_0)$, so $\beta\geq\gamma\geq1$.
We firstly claim that $s_b(W(a, Z))\sim0$.
To verify this, first observe that $s_b^{(2)}(W(a,Z))=\gamma s_{Z_0}(W(a,Z))=2\gamma(t_b(W)-c)$, by (\ref{Z0SymmablesInWaZ}) and (\ref{Z0LongSymmablesInWaZ}).
Writing $W(a, Z)=P_1\circ Q$ and $a^kW^{-1}(a, b)=Q^{-1}\circ P_2$ where $P=P_1\circ P_2$, there exist integers $\lambda_1, \lambda_2, D_1$ with $\lambda_1, \lambda_2\in\{-1, 0, 1\}$ and $|D_1|\leq|P|+2$.
such that:
\begin{align*}
s_b^{(2)}(P_1)+s_b^{(2)}(Q)&=s_b^{(2)}(W(a, Z))+\lambda_1\\
&=2\gamma(t_b(W)-c)+\lambda_1\\
s_b^{(2)}(Q)+s_b^{(2)}(P_2)&= 2\gamma(t_b(W)-c)-s_b^{(2)}(P_1)+s_b^{(2)}(P_2)+\lambda_1\\
s_b^{(2)}(W)&=2\gamma(t_b(W)-c)-s_b^{(2)}(P_1)+s_b^{(2)}(P_2)+\lambda_1-\lambda_2\\
&=2\gamma(t_b(W)-c)+D_1
\end{align*}
Now, we can double-count long $b$-syllables to get that $t_b(W)\geq s_b(W)+s_b^{(2)}(W)= s_b(W)+2\gamma(t_b(W)-c)+D_1$.
Reworking this inequality gives $t_b(W) \sim s_b(W)$,
which gives $s_b^{(2)}(W)\leq t_b(W)- s_b(W)\sim0$. Therefore $s_b^{(2)}(W(a, Z))\sim0$ as follows:
\begin{align*}
s_b^{(2)}(W(a, Z))
&=s_b^{(2)}(P_1)+s_b^{(2)}(Q)-\lambda_1\\
&=s_b^{(2)}(P_1)+s_b^{(2)}(W)-s_b^{(2)}(P_2)-\lambda_1+\lambda_2\\
&\sim0
\end{align*}
Now, by (\ref{Z0SymmablesInWaZ}) and (\ref{Z0LongSymmablesInWaZ}), we have that $s_b(W(a, Z))=\beta s_{Z_0}(W(a, Z))$ and $s_b^{(2)}(W(a, Z))=\gamma s_{Z_0}(W(a, Z))$, to which we can apply $s_b^{(2)}(W(a, Z))\sim0$ to give $\frac{\gamma}{\beta}s_b(W(a, Z))=s_b^{(2)}(W(a, Z))\sim0$. Hence, $s_b(W(a, Z))\sim0$ and our claim is proven.

As (\ref{mainEq}) holds we have $s(W)\sim s_b(W(a, Z))$, which combines with $s_b(W(a, Z))\sim0$ via (\ref{approx1}) to get that $s(W)\sim0$, and the result follows.
\end{proof}

Finally, we consider the case when $s_b^{(2)}(Z_0)=0$ in the form of $Z$.
We require the following fact: If $f(n)$, $g(n)$, $h(n)$ and $k(n)$ are non-negative-valued functions with $f(n)\sim g(n)$, $h(n)\sim k(n)$ and $f(n)\geq h(n)$ then there exists a computable constant $C$ such that
\begin{equation}
g(n)+C\gneq k(n)\label{approx2}
\end{equation}
(here, $C$ is the sum of the constants bounding $|f(n)-g(n)|$ and $|h(n)-k(n)|$).
In the following, the constant $F$ may be computed to be $12|Z|\times|P|$.

\begin{lemma}
\label{lem:SylCase2b}
Suppose that $\varphi_Z$ is injective but not surjective, $Z=Z_0^{-1}a^rZ_0a^q$ for $q, r\neq0$, and $W(a,b)$ is as in (\ref{formOfW}).
Let $s_b^{(2)}(Z_0)=0$.
If (\ref{mainEq}) holds then there exists a constant $F$, computable from $P$ and $Z$, such that $s(W)\leq F$.
\end{lemma}
\begin{proof}
Let $\beta:=s_b(Z_0)$.
We start by approximating the numbers $t_b(W)$ (of occurrences of $b$) and $c$ (of cancelling syllables) in terms of $s_b(W)=n$ (number of $b$-syllables) of $W$:
\begin{align}
t_b(W)&\sim n\label{SylCase2c:id1}\\
c &\sim n \frac{2 \beta -1}{2\beta}\label{SylCase2c:id2}
\end{align}
To verify these, note that by (\ref{Z0SymmablesInWaZ}) and (\ref{Z0LongSymmablesInWaZ}) we have $s_b^{(2)}(W(a, Z))=2s_b^{(2)}(Z_0)(t_b(W)-c)=0$.
As (\ref{mainEq}) holds, $W(a,Z)$ cancels with `most' of $a^{k}W^{-1}(a,b)$,
and so combined with $s_b^{(2)}(W(a, Z))=0$ we have that at most $|P|$ $b$-syllables in $W(a, b)$ have length $\geq2$, and furthermore the sum of the lengths of these long syllables is at most $|P|$.
This implies that $t_b(W)\sim n$, and so (\ref{SylCase2c:id1}) holds.
Now, by (\ref{Z0SymmablesInWaZ}) and (\ref{Z0LongSymmablesInWaZ}), we have that $s_b(W(a, Z))=2\beta(t_b(W)-c)$, and so $s_b(W(a, Z))\sim2\beta(n-c)$ by (\ref{SylCase2c:id1}).
Since the number of $b$-syllables in $W(a,b)$ and $W(a,Z)$ must agree, up to $|P|$ of them, so $s_b(W(a,Z)) \sim s_b(W)=n$, we have $ 2\beta(n-c) \sim n$ via (\ref{approx1}), which rearranges to (\ref{SylCase2c:id2}).

We claim that the sequence of $b$-exponents in $W(a, b)$ contains at most $n-c$ changes of sign, where $n-c\sim\frac{n}{2\beta}$.
To verify this, recall that an $a$-syllable $a^{i_m}$ of $W$ is cancelling only when $j_{m-1}j_m>0$, that is, there is no change of sign in the $b$-syllables preceding and following $a^{i_m}$.
Thus there are $\leq n-c$ changes of sign of the $j_m$'s, and as $c \sim n \frac{2 \beta -1}{2\beta}$ we have $n-c\sim \frac{n}{2\beta}$, as claimed.

Next, we claim that the sequence of $b$-exponents in $W(a, Z)$ contains at least $s_{Z_0}(W(a, Z))-1$ changes of sign, and $s_{Z_0}(W(a, Z))-1\sim\frac{n}{\beta}$.
To verify this, note that in the freely reduced form of $W(a,Z)$ the $Z_0$'s appear in alternate occurrences of $Z_0$ and $Z_0^{-1}$, and so the sequence of $Z_0$-exponents in $W(a, Z)$ contains exactly $s_{Z_0}(W(a, Z))-1$ changes of sign, which gives the required lower bound for the $b$-exponents. To obtain the approximation for $s_{Z_0}(W(a, Z))-1$, note that substituting the identities (\ref{SylCase2c:id1}) and (\ref{SylCase2c:id2}) into (\ref{Z0SymmablesInWaZ}), gives $s_{Z_0}(W(a,Z))\sim \frac{n}{\beta}$, so $s_{Z_0}(W(a,Z))-1\sim \frac{n}{\beta}$, as claimed.

We now prove the lemma.
Write $W(a, Z)=P_1\circ Q$ and $a^kW^{-1}(a, b)=Q^{-1}\circ P_2$ where $P=P_1\circ P_2$.
As above, the sequence of $b$-exponents in $W(a, b)$ contains at most $n-c$ changes of sign, and so the same is true of $a^kW^{-1}(a, b)$, and hence also of $Q$. On the other hand, the sequence of $b$-exponents in $W(a, Z)=P_1Q$ contains at least $s_{Z_0}(W(a, Z))-1$ changes of sign, and so the sequence of $b$-exponents in $Q$ contains at least $s_{Z_0}(W(a, Z))-1-|P|$ changes of sign. Therefore, $n-c\geq s_{Z_0}(W(a, Z))-1-|P|$, and as $n-c\sim\frac{n}{2\beta}$ and $s_{Z_0}(W(a, Z))-1-|P|\sim\frac{n}{\beta}-|P|$, we can apply (\ref{approx2}) to get $\frac{n}{2\beta}+F' \gneq \frac{n}{\beta}-|P|$ for some computable constant $F'$. Setting $F:=2\beta(|P|+F')$, the result follows.
\end{proof}


We now summarise Lemmas \ref{lem:Case1non-cycRed}--\ref{lem:SylCase2b}.
In the following, the constant $B_1$ may be computed to be $12|Z|\times|P|+2|k|$.

\begin{proposition}
\label{prop:BoundSyl}
Suppose that $\varphi_Z$ is injective but not surjective, $Z=Z_0^{-1}Z_1Z_0a^q$ as in (\ref{formOfZ}) and $W(a,b)$ is as in (\ref{formOfW}).
If (\ref{mainEq}) holds then there exists a constant $B_1$, computable from $P$, $Z$ and $k$, such that $s(W)\leq B_1$.
\end{proposition}


\p{Bounding syllable lengths}
Next we bound the length of the individual syllables in any solution $W$ to the relevant instances of the $\varphi_Z$-twisted-conjugacy problem. Combined with Proposition \ref{prop:BoundSyl}, with gives a bound on the number of syllables in $W$, we thus have a bound on $|W|$.


We first need a lemma which says that ``complete'' cancellation is impossible when forming $W(a, Z)a^kW^{-1}$. Note that here, $W(a, b)$ is implicitly non-empty.
\begin{lemma}
\label{lem:completeCancellation}
Suppose that $\varphi_Z$ is injective but not surjective, $Z=Z_0^{-1}Z_1Z_0a^q$ as in (\ref{formOfZ}) and $W(a,b)$ is as in (\ref{formOfW}). Then $\varphi_Z(W)a^kW^{-1}\not\in\langle a\rangle$.
\end{lemma}

\begin{proof}
Suppose $\varphi_Z(W)a^kW^{-1}\in\langle a\rangle$, and we find a contradiction.
Suppose firstly that $t_b(Z_1)\geq 1$. As in the proof of Lemma \ref{lem:SylCase1}, write $W$ as a word $W_0(a, ba^{-q})$ over $a$ and $ba^{-q}$, where $t_b(W)=t_b(W_0)$ and the same working gives us that $0= 2s_b(W_0)t_b(Z_0)+t_b(W)(t_b(Z_1)-1)$, where we have equality rather than inequality because the words corresponding to $U$ and $V$ contain no $b$-syllables.
Now, $2s_b(W_0)t_b(Z_0)\geq0$ so $t_b(W)(t_b(Z_1)-1)\leq0$, and therefore either $t_b(W)=0$ or $t_b(Z_1)=1$ (by assumption, $t_b(Z_1)=0$ cannot happen). By assumption $W$ ends in a $b$-syllable so $t_b(W)=0$ cannot happen, so we have that $t_b(Z_1)=1$, and so $s_b(W_0)t_b(Z_0)=0$. If both $t_b(Z_1)=1$ and $t_b(Z_0)=0$ hold then $\varphi_Z$ is surjective, a contradiction. If $s_b(W_0)=0$ then $t_b(W_0)=0$, so as $t_b(W_0)=t_b(W)$ we have that $t_b(W)=0$, which again is a contradiction.

Hence, we have that $t_b(Z_1)=0$, and so $Z_1=a^r$.
Since $s_b(W(a,b))=n$, we also have $s_b(W(a,Z))=n$, and note that all $b$-syllables in $W(a,Z)$ appear in $Z_0$'s solely (as $Z_1=a^r$). Let $\beta=s_b(Z_0)$.
Now, $Z_0$ in $W(a,Z)$ alternates between $Z_0$ and $Z_0^{-1}$, which implies that after every $\beta$ $b$-syllables in $W(a,Z)$ there must be a change of sign in the exponent of the $b$-syllables (although there might be more changes in total, depending on the structure of $Z_0$), and this behaviour must be mirrored in $W(a,b)$. Recall from the definition of a cancelling syllable that a change of sign from $j_m$ to $j_{m+1}$ means there will be no cancellation between $Z^{j_m}$, $a^{i_m}$ and $Z^{j_{m+1}}$. This will lead to the number of cancellations $c\leq n - \frac{n}{\beta}$, and by (\ref{Z0SymmablesInWaZ}) we get $s_{Z_0}(W(a,Z))\geq 2\left(t_b(W)-\left(n-\frac{n}{t}\right)\right)$, and so $s_b(W(a,Z))=\beta s_{Z_0}(W(a,Z)) \geq 2\beta\left(t_b(W)-\left(n-\frac{n}{\beta}\right)\right)$. As $n=s_b(W(a,Z))$ and $t_b(W)\geq n$, this implies that $n\geq 2\beta \left(n - \left(n - \frac{n}{\beta}\right)\right)=2n$, which gives a contradiction.
\end{proof}


We now give a bound on the lengths of the syllables in $W$. The bound is in terms of $P$, $Z$ and $s(W)$, and so is computable by Proposition \ref{prop:BoundSyl}.

\begin{lemma}
\label{lem:SylLength}
Suppose that $\varphi_Z$ is injective but not surjective, $Z=Z_0^{-1}Z_1Z_0a^q$ as in (\ref{formOfZ}) and $W(a,b)$ is as in (\ref{formOfW}). If (\ref{mainEq}) holds then every syllable of $W$ has length at most $|Z|(s(W)+2)+|P|$.
\end{lemma}

\begin{proof}
%
Note that anytime a syllable of $W^{-1}(a,b)$ is not affected by any cancellations within $W(a,Z)a^kW^{-1}(a,b)$, that syllable clearly has length $\leq|P|$, and if it is affected by a bounded amount $B$, then that syllable has length $\leq B+|P|$.

Suppose first that $q=0$.
We firstly consider the $b$-syllables of $W(a,Z)$, each of which is wholly contained in $Z_0^{-1}Z_1^{j_m}Z_0$ for some $m$, and hence has length at most $|j_m|\times|Z|$. Suppose that $b^{j_l}$ (partially) cancels when we form $W(a,Z)a^kW^{-1}(a,b)$.
As $\varphi_Z$ is non-surjective, $t_b(Z)>1$.
Therefore, if $Z$ consists of a single $b$-syllable then when we form $W(a,Z)a^kW^{-1}(a,b)$ cancellation cannot progress beyond the first $b$-term of $W^{-1}(a,b)$, and so $l=n$ and we easily see that $|j_l|=|j_n|<|Z|+|P|$.
On the other hand, if $Z$ contains an $a$-syllable then $b^{j_l}$ must (partially) cancel with a $b$-syllable contained in $Z_0^{-1}Z_1^{j_m}Z_0$ for some $m<l$, and we see inductively that the largest possible cancellation which occurs in $b^{j_l}$ in this situation is $|Z|\times s(W)$.
Next we consider the $a$-syllable of $W(a,Z)$, each of which is either $a^{i_m}$ for some $m$, or is wholly contained in $Z_0^{-1}Z_1^{j_m}Z_0$ for some $m$. Suppose that $a^{i_l}$ (partially) cancels when we form $W(a,Z)a^kW^{-1}(a,b)$.
As noted above, for an $a$-syllable to cancel when we form $W(a,Z)a^kW^{-1}(a,b)$, $Z$ contains an $a$-syllable. Therefore, $a^{i_l}$ must (partially) cancel with an $a$-syllable contained in $a^{i_m}$ or $Z_0^{-1}Z_1^{j_m}Z_0$ for some $m<l$, and we see inductively that the largest possible cancellation which occurs in $a^{i_l}$ this situation is $|Z|\times s(W)$. In conclusion, any syllable of $W$ is affected by a bounded amount $B=|Z|\times s(W)$ of cancellation, and so every syllable of $W$ has length $\leq |Z|\times s(W)+|P|$, as required.

Now suppose $q \neq 0$. Then the reduced form of $W(a,Z)$ can be seen as a word over $a$, $Z_0$ and $Z_1$, where all $Z_0$-syllables have length $1$ by (\ref{Z0LongSymmablesInWaZ}), and all $Z_1$-syllables have length $\leq s(W)$, which is bounded. To see the latter, note that $Z^k$, for $|k|\geq 2$, has the form $Z_0^{-1} Z_1 Z_0 a^q Z_0^{-1} Z_1 Z_0 a^q \cdots$ (or its inverse), so contains no $Z_1$-syllables of length $\geq 2$, and the only way to obtain longer $Z_1$-syllables is by having consecutive cancelling $a$-syllables, which implies consecutive $j_m$'s (exponents of $b$ in $W$) which are either all $+1$ or all $-1$. Since the number of $j_m$'s is bounded by $s(W)$, the number of consecutive $j_m$'s of value $+1$ or $-1$ is also bounded by $s(W)$, and so any $Z_1$-syllable in $W(a,Z)$ has length $<s(W)$. From this it follows that the length of any $b$-syllable in $W(a,Z)$ is bounded, as it appears entirely within a subword of $Z_0$ and $Z_1$ in $W(a,Z)$, of which the longest have the form $Z_0Z_1^pZ_0$, where $p<s(W)$ as established before. This implies that any $b$-syllable in $W^{-1}(a,b)$ (partially) cancelling with a $b$-syllable in $W(a,Z)$ must have length $<2|Z_0|+s(W)|Z_1|+|P|<|Z|(s(W)+2)+|P|$.

It remains to bound the length of the $a$-syllables in $W$. By the argument above, all $a$-syllables in $W(a,Z)$ which appear within subwords of $Z_0$ and $Z_1$ are bounded, so if an $a$-syllable in $W^{-1}(a,b)$ will (partially) cancel with such an $a$-syllable, it will have length $<|Z|(s(W)+2)+|P|$. Suppose now that an $a$-syllable of $W^{-1}(a,b)$ has length $>2q+|P|$ and cancels with an $a$-syllable in $W(a,Z)$ that is not part of any subword of $Z_0, Z_1$ (or $a^q$); if no such $a$-syllable exists then we are done. Let $a^{-i_m}$ be the syllable satisfying these conditions and where $m$ is largest possible in the set $\{1, \dots, n\}$.
Then $a^{-i_m}$ cancels with some $a^{i_{l}+\lambda_lq}$, $1\leq r\leq n$ and $\lambda_l\in\{-1, 0, 1\}$, within $W(a,Z)$. Now, $m=l$ and to see this first suppose that $l<m$. Then the occurrence of $a^{i_m+\lambda_mq}$ in $W(a,Z)$, $\lambda_m\in\{-1, 0, 1\}$, which is the $a$-syllable of $W(a, Z)$ containing the image of $a^{i_m}$, does not cancel with any syllable in $W^{-1}$, by maximality of $m$, but this syllable must (partially) cancel as $|i_m|>q+|P|$, a contradition. Similarly, if $l>m$ then the $a^{i_l}$ contained in $W(a, b)$ does not cancel when forming $W(a, Z)a^kW^{-1}$, by maximality of $m$, but as $|i_m+\lambda_mq-i_l|<|P|$, and as $i_m>2q+|P|$, we get that $q<i_l$, so this syllable must (partially) cancel, a contradiction.
The $a^{i_l}$ in $W(a,Z)$ and the $a^{-i_l}$ in $W^{-1}$ can cancel only if the appropriate suffixes of $W(a,Z)$ and $W^{-1}$ cancel completely, that is, $a^pW'(a,Z)a^kW'^{-1}(a,b)=1$, where $W'$ is a suffix of $W$ starting with a $b$-syllable, and $p\in \mathbb{Z}$. This is impossible by Lemma \ref{lem:completeCancellation}, and the result follows.
\end{proof}


We now solve the $\varphi_Z$-twisted-conjugacy problem for $P$ and $a^k$, that is, prove Lemma \ref{lem:twistedConjugacy}.

\begin{proof}[Proof of Lemma \ref{lem:twistedConjugacy}]
Recall from the preamble to (\ref{formOfZ}) that $P$ and $a^k$ are $\varphi_Z$-twisted conjugate, for $Z=Z_0^{-1}Z_1Z_0a^q$, if and only if $a^{-q_0}P$ and $a^{k-q_0}$ are $\varphi_{Z_0^{-1}Z_1Z_1a^{q_0+q_1}}$-twisted conjugate. Hence, we may assume that $Z$ has the form (\ref{formOfZ}).

Suppose $P$ and $a^k$ are $\varphi_Z$-twisted-conjugate, so by assumption there exists a word $W \in F(a,b)$ ending in a $b$-syllable such that (\ref{mainEq}) holds. By Proposition \ref{prop:BoundSyl}, there is a bound $B_1$, algorithmically computable from $P$, $Z$ and $k$, on the number of syllables on $W$. By Lemma \ref{lem:SylLength}, there is a bound $B_2$, algorithmically computable from $P$, $Z$ and $B_1$, on the length of the syllables in $W$. Now, the length of a word is simply the sum of the individual syllable lengths, and so $|W|\leq B_1\cdot B_2$. Hence, there is an algorithmically computable bound on $|W|$. Therefore, in order to determine whether or not $P$ and $a^k$ are $\varphi_Z$-twisted-conjugate it is sufficient to check for every word $W$ with $|W|\leq B_1\cdot B_2$ whether or not (\ref{mainEq}) holds. If such a word is found then $P$ and $a^k$ are $\varphi_Z$-twisted-conjugate, otherwise they are not.
\end{proof}


\subsection{From outer fixed points to fixed points}
\label{sec:MainProof}
We now use our solution to instances of the twisted conjugacy problem to prove Theorem \ref{thm:basisexists}.

\begin{lemma}
\label{lem:FindingUW}
Let $\psi\in\emo(F(a, b))$ be an injective, non-surjective endomorphism which has $[a]$ as an outer fixed point. There exists an algorithm with input $\psi$ which determines whether or not $[a]\cap\fix(\psi)\neq \emptyset$, and if $[a]\cap\fix(\psi)\neq \emptyset$ then the algorithm outputs a basis for $\fix(\psi)$.
\end{lemma}

\begin{proof}
Given $\psi$ as in the hypothesis, one can easily find words $P,Q$ such that $\psi({a})=P^{-1}{a}P$ and $\psi({b})=Q$.  Let $Z=PQP^{-1}$.
By Lemma \ref{lem:WordstoFP}, $[a]\cap\fix(\psi)\neq \emptyset$ if and only if there exist $W \in F(a,b)$ and $k\in\mathbb{Z}$ satisfying
 Equation (\ref{mainEq}). By applying Lemma \ref{lem:twistedConjugacy} to Lemma \ref{lem:fromOuterToActualALGORITHM}, we can algorithmically determine the existence of $W$ and $k$, and therefore we can determine whether $[a]\cap\fix(\psi)\neq \emptyset$.
 
If this intersection is non-empty then we can algorithmically compute an element $z$ of $\fix(\psi)$ which is not a proper power (indeed, it is not hard to see that $z:=W({a},{b}){a}W^{-1}({a},{b})$ is such an element, where $W$ is as in the proof of Lemma \ref{lem:FindingUW}). By Lemma \ref{lem:fpUnique}, the subgroup $\fix(\psi)$ is infinite cyclic and so $\fix(\psi)=\langle z\rangle$. Output $z$ as the basis for $\fix(\psi)$. 
 \end{proof}

Finally, we prove Theorem \ref{thm:basisexists}.

\begin{proof}[Proof of Theorem \ref{thm:basisexists}]
Let $\psi\in\emo(F(a, b))$.
Determine, via for example Stallings' foldings, whether or not $\psi$ is injective or surjective.
If $\psi$ is surjective then it is an automorphism and the result is known \cite{Bogopolski2016algorithm}.
If $\psi$ is not injective then Lemma \ref{lem:nonInjCase} produces a basis for $\fix(\psi)$.

If $\psi$ is injective and non-surjective then, by Theorem \ref{thm:OuterFPClassification}, $\psi$ has at most two maximal outer fixed points (up to inversion).
Compute these, via the algorithm of Lemma \ref{lem:algo}.
If $\psi$ has no maximal outer fixed points then $\fix(\psi)=\{1\}$ and so output the empty set as the basis for $\fix(\psi)$.

If $\psi$ has, up to inversion, a single maximal outer fixed point $[x]$ then compute an element $y\in F(a, b)$ such that $\{x, y\}$ forms a basis for $F(a,b)$. Change the basis of $F(a, b)$ from $\{a,b\}$ to $\{x,y\}$. Note that $\fix(\psi)=\{1\}$ if and only if $\fix(\psi)\cap[x]=\emptyset$. Run the algorithm of Lemma \ref{lem:FindingUW} on the free group $F(x,y)$, and if a basis element $z$ of $\fix(\psi)$ is found then output it. Else, $\fix(\psi)=\{1\}$ and so output the empty set as the basis for $\fix(\psi)$.

If $\psi$ has, up to inversion, two maximal outer fixed points $[x]$ and $[y]$ (stored in terms of their representatives $x$ and $y$), then, via the algorithm of Theorem \ref{thm:OuterFPClassification}, compute a representative $\widetilde{y}\in [y]$ such that $\{x, \widetilde{y}\}$ forms a basis for $F(a,b)$. Change the basis of $F(a, b)$ from $\{a,b\}$ to $\{x,\widetilde{y}\}$.
Note that $\fix(\psi)=\{1\}$ if and only if $\fix(\psi)\cap[x]=\fix(\psi)\cap[\widetilde{y}]=\emptyset$. Run the algorithm of Lemma \ref{lem:FindingUW} on the free group $F(x,\widetilde{y})$, and if a basis element $z$ of $\fix(\psi)$ is found then output it. Else, $\fix(\psi)\cap[x]$ is empty so run the algorithm of Lemma \ref{lem:FindingUW} on the free group $F(\widetilde{y}, x)$, and again if a basis element $z$ of $\fix(\psi)$ is found then output it. Else, $\fix(\psi)=\{1\}$ and so output the empty set as the basis for $\fix(\psi)$.
\end{proof}

\bibliographystyle{amsalpha}
\bibliography{BibTexBibliography}

\providecommand{\bysame}{\leavevmode\hbox to3em{\hrulefill}\thinspace}
\providecommand{\MR}{\relax\ifhmode\unskip\space\fi MR }
\providecommand{\MRhref}[2]{%
  \href{http://www.ams.org/mathscinet-getitem?mr=#1}{#2}
}
\providecommand{\href}[2]{#2}
\begin{thebibliography}{MLdASR20}

\bibitem[BH92]{Bestvina1992Traintracks}
Mladen Bestvina and Michael Handel, \emph{Train tracks and automorphisms of
  free groups}, Ann. of Math. (2) \textbf{135} (1992), no.~1, 1--51.
  \MR{1147956}

\bibitem[BH99]{Bridson1999metric}
Martin~R. Bridson and Andr\'{e} Haefliger, \emph{Metric spaces of non-positive
  curvature}, Grundlehren der Mathematischen Wissenschaften [Fundamental
  Principles of Mathematical Sciences], vol. 319, Springer-Verlag, Berlin,
  1999. \MR{1744486}

\bibitem[BM16]{Bogopolski2016algorithm}
Oleg Bogopolski and Olga Maslakova, \emph{An algorithm for finding a basis of
  the fixed point subgroup of an automorphism of a free group}, Internat. J.
  Algebra Comput. \textbf{26} (2016), no.~1, 29--67. \MR{3463201}

\bibitem[BMMV06]{Bogopolski2006Conjugacy}
O.~Bogopolski, A.~Martino, O.~Maslakova, and E.~Ventura, \emph{The conjugacy
  problem is solvable in free-by-cyclic groups}, Bull. London Math. Soc.
  \textbf{38} (2006), no.~5, 787--794. \MR{2268363}

\bibitem[CMZ81]{Cohen1981WhatDoes}
M.~Cohen, Wolfgang Metzler, and A.~Zimmermann, \emph{What does a basis of
  {$F(a,\,b)$} look like?}, Math. Ann. \textbf{257} (1981), no.~4, 435--445.
  \MR{639577}

\bibitem[DKLM19]{Dagstuhl2019}
Volker Diekert, Olga Kharlampovich, Markus Lohrey, and Alexei Myasnikov,
  \emph{Algorithmic problems in group theory}, Dagstuhl seminar report 19131
  (2019),
  \url{http://drops.dagstuhl.de/opus/volltexte/2019/11293/pdf/dagrep_v009_i003_p083_19131.pdf}.

\bibitem[FH94]{Felshtyn1994Reidemeister}
Alexander Fel'shtyn and Richard Hill, \emph{The {R}eidemeister zeta function
  with applications to {N}ielsen theory and a connection with {R}eidemeister
  torsion}, $K$-Theory \textbf{8} (1994), no.~4, 367--393. \MR{1300546}

\bibitem[FH18]{Feighn2018algorithmic}
Mark Feighn and Michael Handel, \emph{Algorithmic constructions of relative
  train track maps and {CT}s}, Groups Geom. Dyn. \textbf{12} (2018), no.~3,
  1159--1238. \MR{3845002}

\bibitem[FT07]{Felshtyn2007TwistedBurnsideFrobenius}
Alexander Fel'shtyn and Evgenij Troitsky, \emph{Twisted {B}urnside-{F}robenius
  theory for discrete groups}, J. Reine Angew. Math. \textbf{613} (2007),
  193--210. \MR{2377135}

\bibitem[FTV06]{Felshtyn2006TwistedBurnsideThm}
Alexander Fel'shtyn, Evgenij Troitsky, and Anatoly Vershik, \emph{Twisted
  {B}urnside theorem for type {$\rm II_1$} groups: an example}, Math. Res.
  Lett. \textbf{13} (2006), no.~5-6, 719--728. \MR{2280770}

\bibitem[Ger87]{Gersten1987Fixed}
S.~M. Gersten, \emph{Fixed points of automorphisms of free groups}, Adv. in
  Math. \textbf{64} (1987), no.~1, 51--85. \MR{879856}

\bibitem[GS91]{Gersten1991Rational}
S.~M. Gersten and H.~B. Short, \emph{Rational subgroups of biautomatic groups},
  Ann. of Math. (2) \textbf{134} (1991), no.~1, 125--158. \MR{1114609}

\bibitem[GSW20]{gonccalves2020twisted}
Daciberg Gon{\c{c}}alves, Parameswaran Sankaran, and Peter Wong, \emph{Twisted
  conjugacy in fundamental groups of geometric $3 $-manifolds},
  arXiv:2003.07791 (2020).

\bibitem[GW09]{Goncalves2009Twisted}
Daciberg Gon\c{c}alves and Peter Wong, \emph{Twisted conjugacy classes in
  nilpotent groups}, J. Reine Angew. Math. \textbf{633} (2009), 11--27.
  \MR{2561194}

\bibitem[Har05]{Hart2005Algebraic}
Evelyn~L. Hart, \emph{Algebraic techniques for calculating the {N}ielsen number
  on hyperbolic surfaces}, Handbook of topological fixed point theory,
  Springer, Dordrecht, 2005, pp.~463--487. \MR{2171115}

\bibitem[IT89]{Imrich1989Endomorphisms}
W.~Imrich and E.~C. Turner, \emph{Endomorphisms of free groups and their fixed
  points}, Math. Proc. Cambridge Philos. Soc. \textbf{105} (1989), no.~3,
  421--422. \MR{985677}

\bibitem[Jia05]{Jiang2005Primer}
Boju Jiang, \emph{A primer of {N}ielsen fixed point theory}, Handbook of
  topological fixed point theory, Springer, Dordrecht, 2005, pp.~617--645.
  \MR{2171118}

\bibitem[JZ18]{Jiang2018Some}
Bo~Ju Jiang and Xue~Zhi Zhao, \emph{Some developments in {N}ielsen fixed point
  theory}, Acta Math. Sin. (Engl. Ser.) \textbf{34} (2018), no.~1, 91--102.
  \MR{3735835}

\bibitem[Kap00]{Kapovich2000Mapping}
Ilya Kapovich, \emph{Mapping tori of endomorphisms of free groups}, Comm.
  Algebra \textbf{28} (2000), no.~6, 2895--2917. \MR{1757436}

\bibitem[Kim16]{Kim2016Twisted}
Seung~Won Kim, \emph{The twisted conjugacy problem for finitely generated free
  groups}, J. Pure Appl. Algebra \textbf{220} (2016), no.~4, 1281--1293.
  \MR{3423447}

\bibitem[LS11]{Ladra2011generalized}
Manuel Ladra and Pedro~V. Silva, \emph{The generalized conjugacy problem for
  virtually free groups}, Forum Math. \textbf{23} (2011), no.~3, 447--482.
  \MR{2805191}

\bibitem[MKS76]{mks}
Wilhelm Magnus, Abraham Karrass, and Donald Solitar, \emph{Combinatorial group
  theory}, revised ed., Dover Publications, Inc., New York, 1976, Presentations
  of groups in terms of generators and relations. \MR{0422434}

\bibitem[MLdASR20]{Araujo2020Twisted}
Paula Macedo Lins~de Araujo and Yuri Santos~Rego, \emph{Twisted conjugacy in
  soluble arithmetic groups}, arXiv:2007.02988 (2020).

\bibitem[MS73]{McCool1973HNN}
James McCool and Paul~E. Schupp, \emph{On one relator groups and {${\rm HNN}$}
  extensions}, J. Austral. Math. Soc. \textbf{16} (1973), 249--256, Collection
  of articles dedicated to the memory of Hanna Neumann, II. \MR{0338186}

\bibitem[{Mut}18]{Mutanguha2018Hyperbolic}
Jean~Pierre {Mutanguha}, \emph{{Hyperbolic Immersions of Free Groups}}, Groups
  Geom. Dyn. (to appear) (2018), arXiv:1809.04761.

\bibitem[Mut19]{mutanguha2019irreducible}
Jean~Pierre Mutanguha, \emph{Irreducible nonsurjective endomorphisms of $ f\_n
  $ are hyperbolic}, Bull. Lond. Math. Soc. (to appear) (2019),
  arXiv:1908.08214.

\bibitem[Mut20]{mutanguha2020dynamics}
\bysame, \emph{The dynamics and geometry of free group endomorphisms},
  arXiv:2005.11896 (2020).

\bibitem[Rey10]{reynolds2010dynamics}
Patrick Reynolds, \emph{Dynamics of irreducible endomorphisms of $ f\_n$},
  arXiv:1008.3659 (2010).

\bibitem[Sta87]{Stallings1987Graphical}
John~R. Stallings, \emph{Graphical theory of automorphisms of free groups},
  Combinatorial group theory and topology ({A}lta, {U}tah, 1984), Ann. of Math.
  Stud., vol. 111, Princeton Univ. Press, Princeton, NJ, 1987, pp.~79--105.
  \MR{895610}

\bibitem[SW20]{sankaran2020twisted}
Parameswaran Sankaran and Peter Wong, \emph{Twisted conjugacy and
  commensurability invariance}, arXiv:2001.02027 (2020).

\bibitem[Tur96]{Turner1996TestWords}
Edward~C. Turner, \emph{Test words for automorphisms of free groups}, Bull.
  London Math. Soc. \textbf{28} (1996), no.~3, 255--263. \MR{1374403}

\bibitem[Ven02]{Ventura2002Fixed}
E.~Ventura, \emph{Fixed subgroups in free groups: a survey}, Combinatorial and
  geometric group theory ({N}ew {Y}ork, 2000/{H}oboken, {NJ}, 2001), Contemp.
  Math., vol. 296, Amer. Math. Soc., Providence, RI, 2002, pp.~231--255.
  \MR{1922276}

\bibitem[Won10]{Wong2010Combinatorial}
Peter Wong, \emph{Combinatorial and geometric group theoretic methods in fixed
  point theory}, Perspectives in geometry and topology, Ramanujan Math. Soc.
  Lect. Notes Ser., vol.~11, Ramanujan Math. Soc., Mysore, 2010, pp.~101--132.
  \MR{2759018}

\bibitem[YK15]{Yi2015Nielsen}
Peter Yi and Seung~Won Kim, \emph{Nielsen numbers of maps of aspherical
  figure-eight type polyhedra}, Forum Math. \textbf{27} (2015), no.~3,
  1277--1307. \MR{3341475}

\end{thebibliography}
\end{document}